\documentclass[bibliography=totoc,final
]{scrartcl}
\usepackage[utf8]{inputenc}
\usepackage{hyperref}
\usepackage[english]{babel}
\usepackage{enumerate}
\usepackage{aliascnt}
\usepackage{amssymb}
\usepackage{amsmath}
\usepackage{verbatim}
\usepackage{mathtools}
\usepackage{nicefrac}

\usepackage{tikz}

\usepackage{amsthm}
\usepackage{cleveref}

\usepackage[notcite,notref]{showkeys}
\usepackage{color}

\usepackage{dsfont} %f\"ur Indikatorfunktion \mathds{1}
\usepackage{hyperref} %Links, autoref

\allowdisplaybreaks[4] %0-4

\newtheorem{prop}{Proposition}[section]

\newaliascnt{lem}{prop} 
\newtheorem{lem}[lem]{Lemma}%[section]

\Crefname{lem}{Lemma}{Lemmas}
\aliascntresetthe{lem} 

\newaliascnt{defn}{prop} 
 \newtheorem{defn}[defn]{Definition}%[section

\aliascntresetthe{defn} 

\newaliascnt{cor}{prop} 
 \newtheorem{cor}[cor]{Corollary}%[section]
 
\aliascntresetthe{cor}

\newaliascnt{rem}{prop} 
 \newtheorem{rem}[rem]{Remark}%[section]
 
\aliascntresetthe{rem} 

\newaliascnt{thm}{prop} 
 \newtheorem{thm}[thm]{Theorem}%[section]
 
\aliascntresetthe{thm} 
\Crefname{thm}{Theorem}{Theorems}

\newaliascnt{hyp}{prop} 
 \newtheorem{hyp}[hyp]{Hypothesis}%[section]
 
\aliascntresetthe{hyp}

\newaliascnt{defnlem}{prop} 
 \newtheorem{defnlem}[defnlem]{Definition and Lemma}%[section]
 
\aliascntresetthe{defnlem}

\Crefname{hyp}{Hypothesis}{Hypotheses}

\newaliascnt{example}{prop} 
 \newtheorem{example}[example]{Example}%[section]
 
\aliascntresetthe{example}

\def\equationautorefname~#1\null{%
  (#1)\null
}

\newcommand{\R}{\ensuremath{\mathbb{R}}}
\newcommand{\N}{\ensuremath{\mathbb{N}}}
\newcommand{\Z}{\ensuremath{\mathbb{Z}}}

\newcommand*\diff{\mathop{}\!\mathrm{d}}

\newcommand{\defeq}{\vcentcolon=}
\newcommand{\eqdef}{=\vcentcolon}
\newcommand{\CalE}{\ensuremath{\mathcal{E}}}
\newcommand{\CalC}{\ensuremath{\mathcal{C}}}

\newcommand{\abs}[1]{\ensuremath{\left \lvert #1\right\rvert}}
\newcommand{\Norm}[2]{\ensuremath{\left\Vert #1 \right\Vert_{#2}}}
\newcommand{\norm}[2]{\ensuremath{\Vert #1 \Vert_{#2}}}
\newcommand{\Image}{\operatorname{Im}}
\newcommand{\codim}{\operatorname{codim}}
\newcommand{\dtzero}{\left.\frac{\diff}{\diff t}\right\vert_{t=0}}
\DeclareMathOperator{\ind}{ind}
\DeclareMathOperator{\Id}{Id}
\DeclareMathOperator{\graph}{graph}
\DeclareMathOperator{\divergence}{div}

\title{On the \texorpdfstring{{\L}ojasiewicz}{Lojasiewicz}--Simon gradient inequality on submanifolds}
\author{Fabian Rupp\thanks{Institute of Analysis, Ulm University, Helmholtzstra\ss e 18, 89081 Ulm, Germany. \texttt{fabian.rupp@uni-ulm.de}.}} 
\begin{document}

\maketitle

\begin{abstract}
\noindent \textbf{Abstract:} We provide sufficient conditions for the {\L}ojasiewicz-Simon gradient inequality to hold on a submanifold of a Banach space and discuss the optimality of our assumptions. Our result provides a tool to study asymptotic properties of quasilinear parabolic equations with (nonlinear) constraints.
\end{abstract}

%\tableofcontents

\bigskip
\noindent \textbf{Keywords:} {\L}ojasiewicz--Simon gradient inequality, constrained gradient flows, mean curvature flow with isoperimetric constraint, constrained Allen--Cahn equation.
 
 \noindent \textbf{MSC(2010)}: 26D10 (primary),  46T05, 37C10  (secondary).
 % 	26D10  	Inequalities involving derivatives and differential and integral operators  cf. Chill
 % 34D05  	Asymptotic properties (Chill)
 % 46T05  	Infinite-dimensional manifolds
 % 	37C10  	Vector fields, flows, ordinary differential equations

%\tableofcontents

 %53C44   	Geometric evolution equations (mean curvature flow, Ricci flow, etc.)
 %35K55   	Nonlinear parabolic equations
 %35K46   	Initial value problems for higher-order parabolic systems
  
 %37K25   	Relations with differential geometry (for dynamical systems)

% 35K52, 53A04, 26D10.   
%35K52  	Initial-boundary value problems for higher-order parabolic systems
%53A04  	Curves in Euclidean space
%26D10  	Inequalities involving derivatives and differential and integral operators

%\tableofcontents

\section{Introduction}
In real algebraic geometry, the \emph{{\L}ojasiewicz inequality} is a remarkable result describing the particular behavior of an analytic function near a critical point. 
\begin{thm}[{\L}ojasiewicz inequality,
	{\cite[Théorème 4]{Loja63}}]\label{thm:LSfinitedim}
	Let $U\subset \R^n$ be open. If $\mathcal{E}\in \mathcal{C}^{\omega}(U;\R)$ and $\bar{u}\in U$ satisfies $\nabla\mathcal{E}(\bar{u})=0$, then there exist $C, \sigma>0$ and $\theta\in (0,\frac{1}{2}]$ such that for all $\norm{u-\bar{u}}{}\leq \sigma$, we have
	\begin{align}\label{eq:LSfinitedim}
	|\mathcal{E}(u)-\mathcal{E}(\bar{u})|^{1-\theta}\leq C \norm{\nabla\mathcal{E}(u)}{}.
	\end{align}
\end{thm}
Throughout this article, we write $\mathcal{C}^{\omega}(U;X)$ for the set of real analytic functions from an open set $U$ of a Banach space $V$ into another Banach space $X$. All vector spaces are understood to be over the field of real numbers $\R$. The space of bounded linear operators between two normed spaces $X$ and $Y$ is denoted by $\mathcal{L}(X,Y)$ and we write  $X^{\ast}\defeq \mathcal{L}(X,\R)$ for the continuous dual of $X$.

In $\R^n$, inequality \eqref{eq:LSfinitedim} was discovered and proven by S. {\L}ojasiewicz in his famous works on semianalytic and subanalytic sets, \cite{Loja63,Loja65}. Since then, \Cref{thm:LSfinitedim} has been used as a celebrated  tool to prove convergence results for the gradient flow of analytic energies on finite-dimensional spaces (see \cite{Loja84}). The pioneering work of L. Simon in \cite{Simon83} extended inequality \eqref{eq:LSfinitedim} to certain energy functions on infinite-dimensional function spaces using \emph{Lyapunov--Schmidt reduction} and, in honor of his significant contributions, the inequality is nowadays often called \emph{{\L}ojasiewicz--Simon gradient inequality.} In more recent work by Kurdyka \cite{Kurdyka98}, {\L}ojasiewicz's convergence result has been extended to a larger class of functions via the \emph{Kurdyka--{\L}ojasiewicz inequality}.
Over the last decades, gradient inequalities like \eqref{eq:LSfinitedim} have been {extensively} studied in various situations to analyze the long time behavior of gradient flows, see for instance { \cite{CFS09,DPS16,FeinSim2000,HaJen99,RybkaHoffmann}}. 
In \cite{Jendoubi981,Jendoubi982}, this is also done for second order evolution equations.
Loosely speaking, whenever an energy $\mathcal{E}$ satisfies a {\L}ojasiewicz--Simon gradient inequality at a critical point {$\bar{u} = \lim_{n\to\infty} u(t_n)$}, where $t_n\to\infty$ and $u=u(t)$ is a precompact solution to the associated gradient flows
\begin{align*}
\left\{\begin{array}{lll}
\partial_t u &= - \nabla\mathcal{E}(u), & t>0\\
u(0)&=u_0,&
\end{array}\right.
\end{align*}
we may conclude that $u$ converges with {$\lim_{t\to\infty}u(t) = \bar{u}$}. Numerical applications of this phenomenon have been considered  for instance in \cite{Andrews05,HBR10}.

Hence, it is a question of great interest, whether a given energy function satisfies a {\L}ojasiewicz--Simon gradient inequality. It can be shown that in the infinite-dimensional case, mere analyticity of the energy is not enough, see for instance \cite[Theorem 2.1, Proposition 3.5]{Haraux11}. On the other hand, very general conditions which are sufficient for the gradient inequality to hold are presented in \cite{Chill03}.  

For most of the applications, one usually checks that the following conditions are satisfied, see \cite{CFS09,DPS16,Feehan15,Lengeler}. 
\begin{thm}[{Consequence of \cite[Corollary 3.11]{Chill03}}]\label{thm:FredholmLoja}
	Let $V$ be a Banach space, $U\subset V$ an open set, $\mathcal{E}\in \mathcal{C}^{\omega}(U;\R)$ and $\bar{u}\in U$ a critical point of $\mathcal{E}$. Suppose that
	\begin{enumerate}[(i)]
		\item there exists a Banach space $Z$ such that $V\hookrightarrow Z$ densely,
		\item $\mathcal{E}^{\prime}\in \mathcal{C}^{\omega}(U;Z^{\ast})$,
		\item the second derivative $\mathcal{E}^{\prime\prime}(\bar{u})\colon V\to Z^{\ast}$ is Fredholm of index zero.
	\end{enumerate}
	Then, there exist $C, \sigma>0$, $\theta\in (0,\frac{1}{2}]$ such that for all $u\in U$ with $\Norm{u-\bar{u}}{V}\leq \sigma$, we have
	\begin{align}
	\label{eq:LSlin}
	\abs{\mathcal{E}(u)-\mathcal{E}(\bar{u})}^{1-\theta}\leq C \Norm{\mathcal{E}^{\prime}(u)}{Z^{\ast}}.
	\end{align}
\end{thm}
\begin{rem}
	Note that by assumption (i) in \Cref{thm:FredholmLoja} we have $V\hookrightarrow Z$, so $Z^{\ast}$ can be identified with a subset of $V^{\ast}$. Condition (ii) requires that for all $u\in U$ the functional $\mathcal{E}^{\prime}(u)$ which is in general only in $V^{\ast}$ is in fact in $Z^{\ast}$ and the map $\mathcal{E}^{\prime}\colon U\to Z^{\ast}$ is analytic.
\end{rem}
Although \Cref{thm:FredholmLoja} describes a slightly less general situation than in \cite{Chill03}, in most applications its conditions are relatively easy to check and suffice to prove the {\L}ojasiewicz--Simon gradient inequality. The details on how to deduce \Cref{thm:FredholmLoja} from \cite{Chill03} are given in \Cref{subsec:AppendixLoja}.

To prove a suitable version of \Cref{thm:LSfinitedim} on a finite-dimensional manifold $\mathcal{M}$ is quite straightforward if $\mathcal{M}$ and $\mathcal{E}$ are analytic, by simply choosing local coordinates and applying \Cref{thm:LSfinitedim}. In \cite{Lageman07}, this is used to study \emph{gradient-like} dynamical systems via the Kurdyka--{\L}ojasiewicz inequality. The infinite-dimensional setting is more complicated.

Our main result is to extend \Cref{thm:FredholmLoja} to a constrained energy function $\mathcal{E}\vert_{\mathcal{M}}$ on a submanifold $\mathcal{M}$ of a Banach space $V$, and to refine the estimate by projecting the derivative onto the cotangent space of $\mathcal{M}$. In \cite{Lengeler}, a special case has been studied and a {\L}ojasiewicz--Simon gradient inequality is proven for the \emph{Canham--Helfrich energy} on the submanifold of closed embedded surfaces with fixed area and volume, see \cite[Theorem 1.4]{Lengeler}. In the following theorem, we give very general sufficient conditions for the {\L}ojasiewciz--Simon gradient inequality to hold on an infinite-dimensional submanifold in the abstract setting of an energy on a Banach space. In \Cref{sec:HilbertSpace}, we will consider the easier case where the ambient space is a Hilbert space. However, as we shall explain in detail in \Cref{rem:HilbertBad} below, in order to avoid issues with analyticity, it is sometimes necessary to work in Banach spaces, cf. also \Cref{subsec:isoperimetric}. Our main result is the following
\begin{thm}\label{thm:main}
	Let $V$ be a Banach space, $U\subset V$ an open set, $m\in\N$ and $\mathcal{E}\colon U \to\R$, $\mathcal{G}\colon U\to\R^m$ be analytic. Let $\bar{u}\in U$ and suppose that
	\begin{enumerate}[(i)]
		\item there exists a Banach space $Y$ such that $V\hookrightarrow Y$ densely,
		\item $\mathcal{E}^{\prime}\in \mathcal{C}^{\omega}(U;Y^{\ast})$,
		\item the second derivative $\mathcal{E}^{\prime\prime}(\bar{u}) \colon V\to Y^{\ast}$ is Fredholm of index zero,
		\item for any $u\in U$, the linear operator $\mathcal{G}^{\prime}(u)\in \mathcal{L}(V,\R^m)$ extends to $\overline{\mathcal{G}^{\prime}(u)} \in \mathcal{L}(Y, \R^m)$ and the map $\overline{\mathcal{G}^{\prime}}\colon U\to \mathcal{L}(Y,\R^m)$, $u\mapsto\overline{\mathcal{G}^{\prime}(u)}$ is analytic,
		\item the Fréchet derivative $\big(\overline{\mathcal{G}^{\prime}}\big)^{\prime}(\bar{u})\colon V\to\mathcal{L}(Y,\R^m)$ is compact,
		\item $\mathcal{G}(\bar{u})=0$ and $\mathcal{G}^{\prime}(\bar{u})\colon V\to \R^m$ is surjective.
	\end{enumerate}
	Then, $\mathcal{M} \defeq \{ u\in U\mid \mathcal{G}(u)=0\}$ is locally an analytic submanifold of $V$ of codimension $m$ near $\bar{u}$. 
	
	If $\bar{u}$ is a critical point of $\mathcal{E}\vert_{\mathcal{M}}$, then the restriction satisfies a \emph{refined {\L}ojasiewicz--Simon gradient inequality} at $\bar{u}$, i.e. there exist $C, \sigma>0$ and $\theta\in (0,\frac{1}{2}]$ such that for any $u\in \mathcal{M}$ with $\Norm{u-\bar{u}}{V}\leq \sigma$, we have
	\begin{align}\label{eq:LSnonlin}
	\abs{\mathcal{E}(u)-\mathcal{E}(\bar{u})}^{1-\theta}\leq C \Norm{\mathcal{E}^{\prime}(u)}{\overline{\mathcal{T}_u\mathcal{M}}^{\ast}}.
	\end{align}
	Here, $\overline{\mathcal{T}_u\mathcal{M}}^{\ast}$ is the dual of the closure $\overline{\mathcal{T}_u\mathcal{M}}\defeq \overline{\mathcal{T}_u\mathcal{M}}^{\Norm{\cdot}{Y}}\subset Y$ of the tangent space $\mathcal{T}_u\mathcal{M}$. 
\end{thm}
\begin{rem}
	The notation $\overline{\mathcal{G}^{\prime}(u)}$ is justified, since the operator $\overline{\mathcal{G}^{\prime}(u)}\colon Y\to \R^m$ is the closure of $A=\mathcal{G}^{\prime}(u)$ on the Banach space $Y$ with $D(A) = V$.
\end{rem}
\begin{rem}\begin{enumerate}[(i)]
		\item Note that we could apply \Cref{thm:FredholmLoja} in the situation of \Cref{thm:main} as well, but \eqref{eq:LSnonlin} yields a sharper estimate: If $Z=Y$ with $Y$ as in \Cref{thm:main}, then for $u, \bar{u}\in \mathcal{M}$ with $\Norm{u-\bar{u}}{V}\leq \sigma$, we have
		\begin{align*}
		\Norm{\mathcal{E}^{\prime}(u)}{\overline{\mathcal{T}_u\mathcal{M}}^{\ast}} &=\sup_{0\neq y\in \overline{\mathcal{T}_u\mathcal{M}}} \frac{\mathcal{E}^{\prime}(u)y}{\norm{y}{Y}} \leq  \sup_{0\neq y\in \overline{\mathcal{T}_u\mathcal{M}}} \frac{\Norm{\mathcal{E}^{\prime}(u)}{Y^{\ast}}\Norm{y}{Y}}{\norm{y}{Y}} = \Norm{\mathcal{E}^{\prime}(u)}{Y^{\ast}}.
		\end{align*}
		Thus, if the assumptions of \Cref{thm:main} are satisfied and $C, \sigma, \theta$ are as in \Cref{thm:main}, we have $\abs{\mathcal{E}(u)-\mathcal{E}(\bar{u})}^{1-\theta}\leq C\Norm{\mathcal{E}^{\prime}(u)}{Y^{\ast}}$, i.e. \eqref{eq:LSnonlin} implies \eqref{eq:LSlin} under the assumptions of \Cref{thm:main}. It hence makes sense to refer to \eqref{eq:LSnonlin} as a \emph{refined {\L}ojasiewicz--Simon gradient inequality}.
		\item From our proof, we cannot conclude that the \emph{{\L}ojasiewicz exponents} $\theta$ in \Cref{thm:FredholmLoja} and \Cref{thm:main} coincide. %However, it seems reasonable to believe so.
	\end{enumerate}
\end{rem}

\begin{rem}\label{rem:HilbertBad}
	The Hilbert space case treated in \Cref{cor:main hilbert} is much easier to handle than \Cref{thm:main}. It is also more natural since one usually studies $H$-gradient flows with $H = W^{k,2}(\Omega)$, $\Omega\subset \R^d$ open, $k\in \Z$. On the other hand, one may sometimes encounter a problem in proving analyticity of the energy. The problematic phenomenon is, that whenever a \emph{Nemytskii} or \emph{supercomposition operator} 
	\begin{align*}
	\mathcal{F}\colon L^p(\Omega)\to L^q(\Omega),\quad \mathcal{F}(v) = f(v)= f\circ v\text{ with } p,q \in [1,\infty)
	\end{align*}
	is analytic, the function $f$ has to be a polynomial of degree at most $\left\lceil {\frac{p}{q}}\right\rceil$, see \cite[Theorem 3.16]{Appell90}. A way to work around this, is to choose suitable Sobolev spaces, such that all derivatives in the energy either appear in polynomial expressions with appropriate powers or are continuous. This is exactly why we work in the Banach space $W^{2,p}(\Omega)$ with $p>d$ to prove the {\L}ojasiewicz--Simon gradient inequality in \Cref{subsec:isoperimetric}.
\end{rem}

This article is structured as follows. First, we recall some basic definitions and fundamental properties of analytic functions and Fredholm operators. Then we present the generalizations of basic concepts of differential geometry to submanifolds of a Banach space.
In \Cref{sec:chartgraph}, we establish a local graph representation for the manifold $\mathcal{M}$ in \Cref{thm:main}. It turns out that studying this chart plays a crucial role in the proof of \Cref{thm:main} which we complete in \Cref{sec:ProofMain}. After that, we consider the Hilbert space case in \Cref{sec:HilbertSpace} in which the inequality takes a more convenient form. {We also prove an abstract convergence result for the associated gradient flow in this case}. \Cref{sec:Optimality} is dedicated to discuss the necessity of the assumptions we make in \Cref{thm:main}.  In the last section, we will then apply our abstract results to the area of graph surfaces with an isoperimetric constraint in \Cref{subsec:isoperimetric}, the Allen--Cahn equation in \Cref{subsec:AllenCahn} and to surfaces of revolution with prescribed volume in \Cref{subsec:CMCRevolution}.
\section{Preliminaries}
\subsection{Analyticity}
\begin{defn}\label{def:analytic}
	Let $V, W$ be (real) Banach spaces, $D\subset V$ be an open set. A function $f\colon D\to W$ is called \emph{(real) analytic at $u_0\in D$} if there exist $\rho>0$ and continuous $\R$-multilinear forms $a_n\colon V^n\defeq \underbrace{V\times
		\dots \times V}_{n-\text{times}}\to W$ for all $n\in \N_0$ such that
	\begin{align}\label{eq:analytic}
	\sum_{n=0}^{\infty} \norm{a_n}{\mathcal{L}(V^n,W)} \Norm{u-u_0}{V}^{n} \text{ converges and } f(u) = \sum_{n=0}^{\infty} a_n(u-u_0)^{n} \text{ in } W
	\end{align}
	for all $\Norm{u-u_0}{V}<\rho$, where $a_n(u-u_0)^{n}\defeq a_n(u-u_0, \dots, u-u_0)\in W$. The function $f$ is \emph{(real) analytic (on $D$)} if it is analytic at every point $u_0\in D$.
\end{defn}
We denote by $\mathcal{C}^{\omega}(D;W)$ the vector space of analytic functions from $D$ to $W$.
Like in the finite-dimensional case, a composition of two analytic maps is analytic.
\begin{thm}[{\cite[p. 1079]{Whittsley65}}]\label{thm:companalytic}
	Let $V,W,X$ be Banach spaces, $D\subset V$ and $E\subset W$ be open and $f\colon D\to W$, $g\colon E\to X$ be analytic with $f(D)\subset E$. Then $g\circ f\colon D\to X$ is analytic.
\end{thm}
Easy examples of analytic maps are bounded multilinear maps.
\begin{example}\label{ex:multilinAnalytic}
	Let $\ell \in \N$ and $V_1,\dots, V_{\ell}, W$ be Banach spaces. If $a\colon V_1\times \dots\times V_{\ell}\to W$ is multilinear and continuous, then it is analytic. This follows easily since the series in \eqref{eq:analytic} consists of exactly one nonzero term and hence converges.
\end{example}
%Together with \Cref{thm:companalytic}, this can be used to show that $\mathcal{C}^{\omega}(U,W)$, the set of analytic functions is vector space, i.e. linear combinations of analytic functions are again analytic.
\subsection{Fredholm operators}
\begin{defn}
	Let $V, W$ be Banach spaces. An operator $T\in \mathcal{L}(V,W)$ is called a \emph{Fredholm operator} if both $\dim\ker T$ and $\codim(\Image T, W) = \dim(W/\Image T)$ are finite. The number $\ind{T} \defeq \dim\ker T - \codim(\Image T,W)$
	is called the \emph{Fredholm index of $T$}.
\end{defn}
In the following, we collect some important properties of Fredholm operators.
\begin{prop}[{\cite[XVII, Corollaries 2.6 and 2.7]{Lang12}}]\label{prop:Fredholmproperties}
	Let $T\in \mathcal{L}(V,W)$ be a Fredholm operator. Then
	\begin{enumerate}[(i)]
		\item the image $\Image T\subset W$ is closed,
		\item for any compact operator $K\colon V\to W$, the perturbed operator $T+K$ is Fredholm with $\ind{(T+K)}=\ind{T}$. This holds in particular if $K$ has finite rank.
	\end{enumerate}
\end{prop}
\begin{thm}[{\cite[XVII, Theorem 2.8]{Lang12}}]\label{thm:FredholmIndexFormula}
	Let $V,W$ and $X$ be Banach spaces and let $T\in \mathcal{L}(V,W)$ and $S\in \mathcal{L}(W,X)$ be Fredholm operators. Then $S\circ T\in \mathcal{L}(V, X)$ is a Fredholm operator and its index is given by $\ind{(S\circ T)}= \ind{S}+\ind{T}$.
\end{thm}
\subsection{Complemented subspaces}
Projection operators and complemented subspaces play a crucial role in the proof of the {\L}ojasiewicz--Simon gradient inequality in \cite{Chill03} and they will also be important for our result, specifically when investigating the properties of the submanifold $\mathcal{M}$ in \Cref{thm:localgraph}.
\begin{defnlem}
	A closed subspace $V_0$ of a Banach space $V$ is called \emph{complemented in $V$} if there exists a projection $P\in \mathcal{L}(V)$ with $\Image P = V_0$. Equivalently, there exists a closed subspace $V_1$ of $V$ with $V = V_0\oplus V_1$, see \cite[Section 2.4]{Brezis10}.
\end{defnlem}
Whereas in a Hilbert space, every closed subspace is complemented via the orthogonal projection (cf. \cite[Chapter 5.1]{Brezis10}), this is not true for a general Banach space. In fact, if in a Banach space $V$, every closed subspace is complemented, then  it has to be isomorphic to a Hilbert space, see \cite{LindenstraussTzafriri}. Nevertheless, some subspaces are always complemented.
\begin{lem}[{\cite[XV, Corollary 1.6]{Lang12}}]\label{lem:fin_co_dim complemented}
	Let $V$ be a Banach space and $V_0\subset V$ be a closed subspace, such that $\dim V_0<\infty$ or
	$\codim(V_0,V)<\infty$. Then $V_0$ is complemented in $V$.
\end{lem}
\subsection{Submanifolds of Banach spaces}
This section is devoted to review some basic definitions in differential geometry in the setting of infinite-dimensional manifolds. Since we are only interested in the case of a submanifold of a Banach space $V$, the following definition based on \cite[Definition 3.2.1]{Abraham04} is sufficient for our purposes.
\begin{defn}\label{def:manifold}
	Let $V$ be a Banach space. A subset $\mathcal{M}\subset V$ is called a \emph{(splitting) submanifold of $V$ (of class $\mathcal{C}^{\ell}$)} if for all $u\in \mathcal{M}$, there exists an open neighborhood $U\subset V$ of $u$, a complemented subspace $V_0\subset V$ and a  map $\alpha\in\mathcal{C}^{\ell}(U;V)$ which is a diffeomorphism onto its image, such that $\alpha (U\cap \mathcal{M}) = \alpha(U)\cap V_0.$ 
	If $\alpha \in \mathcal{C}^{\omega}(U;V)$, we say that $\mathcal{M}$ is \emph{analytic}.
\end{defn}
\begin{example}\label{ex:graphManifold}
	If $V$ is a Banach space, $V_0\subset V$ is a complemented subspace, with $V=V_0\oplus V_1$, $\Omega_0\subset V_0$ is an open set and $\psi\in \mathcal{C}^{\ell}(\Omega_0;V)$ with $\psi(\Omega_0)\subset V_1$, then $\mathcal{M}\defeq \{ \omega+\psi(\omega)\mid \omega\in \Omega_0\}$ is a submanifold of $V$ of class $\mathcal{C}^{\ell}$.
	
	Indeed, let $\Omega\defeq \Omega_0+V_1$ and write $\Omega\ni v = \omega + v_1$ with $\omega\in \Omega_0$ and $v_1\in V_1$ and define $\alpha \colon \Omega \to V, \alpha(\omega+ v_1) = \omega+ (v_1-\psi(\omega))\in V_0\oplus V_1$. Then $\alpha$ is of class $\mathcal{C}^{\ell}$ and \begin{align*}
	\alpha^{\prime}(\omega)\begin{bmatrix}
	v_0\\ v_1
	\end{bmatrix}&= \begin{bmatrix}
	\Id_{V_0} & 0 \\
	-\psi^{\prime}(x)&  \Id_{V_1}\\
	\end{bmatrix}
	\begin{bmatrix}
	v_0\\v_1
	\end{bmatrix} \text{ for all } \omega\in \Omega_0, v_0+v_1\in V_0\oplus V_1,
	\end{align*}
	so $\alpha^{\prime}(\omega)\colon V\to V$ is an isomorphism. Since $\alpha$ is clearly bijective onto its image, we conclude that $\alpha$ is a $\mathcal{C}^{\ell}$-diffeomorphism by the Inverse Function Theorem \cite[XIV, Theorem 1.2]{Lang12}. Consequently,
	\begin{align*}
	\alpha(\Omega\cap M) = \alpha(\{  \omega+\psi(\omega)\mid \omega\in \Omega_0\}) = \{ \omega+\psi(\omega)-\psi(\omega)\mid \omega\in \Omega_0\} = \alpha(\Omega)\cap V_0,
	\end{align*}
	thus $\mathcal{M}$ is a submanifold in the sense of \Cref{def:manifold}.
\end{example}
\begin{defn}\label{def:tangentSpace}
	Let $\mathcal{M}\subset V$ be a submanifold of class $\mathcal{C}^{\ell}$ with $\ell \geq 1$. The \emph{tangent space $\mathcal{T}_u\mathcal{M}$ of $\mathcal{M}$ at $u\in \mathcal{M}$} is defined by
	\begin{align*}
	\mathcal{T}_u\mathcal{M}\defeq \left\{ \gamma^{\prime}(0)\mid \exists \varepsilon>0, \gamma\in \mathcal{C}^1\big((-\varepsilon, \varepsilon);V\big) \text{ with } \gamma(t)\in \mathcal{M}~\forall t\in (-\varepsilon, \varepsilon) \text{ and }\gamma(0)=u\right\}.
	\end{align*}
	Like in the finite-dimensional case, $\mathcal{T}_u\mathcal{M}\subset V$ is a subspace. We define the \emph{codimension of $\mathcal{M}$ in $V$} to be the codimension $\codim(\mathcal{T}_u\mathcal{M}, V)$ of $\mathcal{T}_u\mathcal{M}$ in $V$.
	The dual of the tangent space is called \emph{cotangent space} and denoted by $\mathcal{T}_u^{\ast}\mathcal{M}\defeq \left(\mathcal{T}_u\mathcal{M}\right)^{\ast}$.
\end{defn}

\begin{defnlem}\label{prop:CharaConstCritPoint}
	Let $V$ be a Banach space, $U\subset V$ be an open set, $\emptyset\neq M\subset U$ and $\mathcal{E}\in \mathcal{C}^1(U;\R)$. We say that $\bar{u}$ is a \emph{constraint critical point of $\mathcal{E}$ on $M$} or a \emph{critical point of $\mathcal{E}\vert_\mathcal{M}$}, if for any curve $\gamma \in \mathcal{C}^{1}\big((-\varepsilon, \varepsilon);V\big)$ with $\gamma(0)=\bar{u}$ and $\gamma(t)\in M$ for all $t\in (-\varepsilon, \varepsilon)$, the map $t\mapsto(\mathcal{E}\circ\gamma)(t)$ has a critical point at $t=0$.
	
	If $M=\mathcal{M}\subset V$ is a submanifold, then $\bar{u}\in \mathcal{M}$ is a constraint critical point if and only if
	\begin{align*}%\label{eq:CharaConstCritPoint}
	\mathcal{E}^{\prime}(\bar{u}) v = 0 \text{ for all } v\in \mathcal{T}_{\bar{u}}\mathcal{M}\subset V.
	\end{align*}
\end{defnlem}
\begin{proof}
	This follows since for each curve $\gamma\in \mathcal{C}^1((-\varepsilon, \varepsilon); V)$ with $\gamma(0)=\bar{u}$ and $\gamma(t)\in \mathcal{M}$ for all $t\in (-\varepsilon, \varepsilon)$, we have $0 = \dtzero (\mathcal{E}\circ\gamma)(t) = \mathcal{E}^{\prime}(\bar{u}) \gamma^{\prime}(0)$.
\end{proof}
%If the Banach space $V$ is endowed with an inner product structure, there is a natural way to turn $\mathcal{M}$ into a (weak) Riemannian manifold.
%\begin{prop}
%Suppose $\langle \cdot, \cdot\rangle$ is an inner product on $V$ and $\mathcal{M}\subset V$ is a submanifold. Then, the induced product
%\begin{align*}
%g_u(v,w) \defeq \langle v, w\rangle \text{ for }u\in \mathcal{M}, v,w\in \mathcal{T}_u\mathcal{M}\subset V
%\end{align*}
%defines a \emph{weak Riemannian metric} on $\mathcal{M}$ (cf. \cite[Definition 5.2.12]{Abraham04}). If $(V,\langle \cdot, \cdot\rangle)$ is a Hilbert space and $\mathcal{T}_u\mathcal{M}$ is closed for each $u\in \mathcal{M}$, then $g$ is a \emph{strong Riemannian metric}, i.e. $v\mapsto g_u(v,\cdot)\colon \mathcal{T}_u\mathcal{M}\to\mathcal{T}^{\ast}_u\mathcal{M}$ defines an isomorphism for each $v\in \mathcal{T}_u\mathcal{M}$.
%\end{prop}
%\begin{proof}
%If $(V, \langle\cdot,\cdot\rangle)$ is a Hilbert space, then so is $\mathcal{T}_u\mathcal{M}$ for each $u$, since it is closed. The claim then follows from the Riesz-Fréchet theorem \cite[V, Theorem 2.1]{Lang12}.
%\end{proof}

\section{Local representation by a graph}\label{sec:chartgraph}
In this section, we will lay the foundations for the proof of our main theorem. We will see that the level set manifold $\mathcal{M}$ in \Cref{thm:main} admits a natural chart around $\bar{u}$ representing $\mathcal{M}$ locally as a graph. After that, we will carefully analyze the properties of this induced chart.

For the rest of the article, we assume that $V$ and $Y$ are Banach spaces with $V\hookrightarrow Y$ densely, thus we get an induced embedding $Y^{\ast}\hookrightarrow V^{\ast}$. Furthermore, we assume that $U\subset V$ is an open set, $m\in \N$ and $\mathcal{G}\colon U\to\R^m$  is analytic. We study the nodal set of $\mathcal{G}$ given by $\mathcal{M} \defeq \{u\in U\mid \mathcal{G}(u)=0\}$.
\begin{thm}\label{thm:localgraph}
	Let $\bar{u}\in \mathcal{M}$ such that $\mathcal{G}^{\prime}(\bar{u})\colon V\to\R^m$ is surjective. Then $V=V_0\oplus V_1$ with $V_0=\ker\mathcal{G}^{\prime}(\bar{u})$ for a closed subspace $V_1\subset V$. Moreover, there exist open sets $ \Omega_0\subset V_0, \Omega_1\subset V_1$ with $\bar{u}\in \Omega = \Omega_0\times\Omega_1 \subset U$ and an analytic function $\psi \colon\Omega_0\to V$ with $\psi(\Omega_0)=\Omega_1$ such that
	\begin{align*}
	\mathcal{M}\cap\Omega = \{ \omega+\psi(\omega)\mid \omega\in \Omega_0\}.
	\end{align*} Hence, locally around $\bar{u}$, $\mathcal{M}$ is an analytic submanifold of $V$.
	Moreover, with $\varphi\colon\Omega_0\to V,$ $\varphi(\omega)\defeq \omega+\psi(\omega)$ we have for any $\omega\in \Omega_0, v\in V_0$
	\begin{align}\label{eq:psi'}
	\psi^{\prime}(\omega)v  &=-\left(\frac{\partial\mathcal{G}}{\partial v_1}(\varphi(\omega))\right)^{-1}\circ\mathcal{G}^{\prime}(\varphi(\omega))v, \\
	\varphi^{\prime}(\omega)v &= v-\left(\frac{\partial\mathcal{G}}{\partial v_1}(\varphi(\omega))\right)^{-1}\circ\mathcal{G}^{\prime}(\varphi(\omega))v, \label{eq:varphi'}
	\end{align}
	where $\frac{\partial \mathcal{G}}{\partial v_1}(u)\defeq  \mathcal{G}^{\prime}(u)\vert_{V_1}\colon V_1 \to \R^m$ for $u\in U$.
\end{thm}
\begin{proof}
	Since $\R^m=\Image \mathcal{G}^{\prime}(\bar{u}) \cong {V}/{V_0}$, $V_0$ has finite codimension in $V$. Moreover, since $V_0$ is closed by continuity, it is complemented by \Cref{lem:fin_co_dim complemented}, i.e. there exists $V_1\subset V$ closed with $V= V_0\oplus V_1$. As a consequence thereof, $\frac{\partial \mathcal{G}}{\partial v_1}(\bar{u}) \colon V_1 \to \R^m$ is an isomorphism of Banach spaces. Thus, by the Implicit Function Theorem \cite[Theorem 4.B]{Zeidler86}, there exist open neighborhoods $\Omega_0\subset V_0, \Omega_1 \subset V_1$ with $\Omega\defeq\Omega_0\times \Omega_1 \subset U$ and $u\in \Omega$ such that for any $\omega\in \Omega_0$, there exists exactly one $\psi(\omega)\in \Omega_1$ with $\mathcal{G}(\omega+\psi(\omega))=0$. Analyticity of $\psi$ follows since $\mathcal{G}$ is analytic.
	By \Cref{ex:graphManifold}, we may conclude that $\mathcal{M}\cap\Omega$ is an analytic submanifold of $V$.
	
	Moreover, the subset of invertible operators in $\mathcal{L}(V_1,\R^m)$ is open in the norm topology and the map $\Omega \ni u\mapsto\mathcal{G}^{\prime}(u)\vert_{V_1}\in \mathcal{L}(V_1,\R^m)$ is continuous. Hence, by continuity, we can assume that $\frac{\partial \mathcal{G}}{\partial v_1}(u)\colon V_1\to\R^m$ is an isomorphism for all $u\in \Omega$, passing to a smaller $\Omega$ if necessary.
	Therefore, \eqref{eq:varphi'} and thus \eqref{eq:psi'} follow by differentiating the equation $0=\mathcal{G}(\omega+\psi(\omega))=\mathcal{G}(\varphi(\omega))$ for $\omega\in \Omega_0$.
\end{proof}
\begin{rem}\label{rem:Mmanifold}
	\begin{enumerate}[(i)]
		\item The relation $\mathcal{M} \cap \Omega = \{ \omega+\psi(\omega) \mid \omega\in \Omega_0\}$ implies that the map $\varphi\colon\Omega_0\to \Omega\cap\mathcal{M}, \varphi(\omega)=\omega+\psi(\omega)$ defines a \emph{chart for $\mathcal{M}\cap\Omega$ centered at $\bar{u}\in \mathcal{M}$}. Resembling the finite-dimensional case, we can identify $\omega+\psi(\omega) = (\omega, \psi(\omega))$ which means that $\mathcal{M}$ is locally the graph of $\psi$ near $\bar{u}$ (cf. \Cref{ex:graphManifold}).
		\item Since we only work locally, we will abuse notation and speak about the manifold $\mathcal{M}$ instead of $\mathcal{M}\cap\Omega$ and write $\mathcal{T}_{u}\mathcal{M}$ for the tangent space $\mathcal{T}_u(\mathcal{M}\cap\Omega)$ at $u$.
	\end{enumerate}
\end{rem}
The assumptions on $\mathcal{G}$ in \Cref{thm:main} have some immediate consequences for the tangent space of $\mathcal{M}$.
\begin{prop}\label{prop:propertiesTuM}
	Suppose $\mathcal{G}\colon U\to\R^m$ and $\bar{u}\in \mathcal{M}$ satisfy assumptions (i), (iv) and (vi) in \Cref{thm:main}. Then, using the notation of \Cref{thm:localgraph}, for $\omega\in \Omega_0, \varphi(\omega)=u$ we have
	\begin{enumerate}[(i)]
		\item $\mathcal{T}_u\mathcal{M} = \ker\mathcal{G}^{\prime}(u)=\Image\varphi^{\prime}(\omega)$,
		\item $\overline{\mathcal{T}_u\mathcal{M}}\defeq\overline{\mathcal{T}_u\mathcal{M}}^{\Norm{\cdot}{Y}} = \ker\overline{\mathcal{G}^{\prime}(u)}$,
		\item $\codim(\mathcal{T}_u\mathcal{M}, V) = \codim(\overline{\mathcal{T}_u\mathcal{M}}, Y)=m$ for all $u\in \Omega$.
	\end{enumerate}
\end{prop}
\begin{proof}
	\begin{enumerate}[(i)]
		\item We first prove the inclusion $\mathcal{T}_u\mathcal{M}\subset\ker\mathcal{G}^{\prime}(u)$. Let $v\in \mathcal{T}_u\mathcal{M}$. Then, there exist $\varepsilon>0$ and $\gamma \in\mathcal{C}^1((-\varepsilon, \varepsilon), V)$ with $\Image\gamma \subset \mathcal{M}\cap\Omega, \gamma(0)=u$ and $\gamma^{\prime}(0)=v$. Now, $\mathcal{G}^{\prime}(u) v = \dtzero \mathcal{G}(\gamma(t)) =0$, since $\mathcal{M} = \{ u\in U\mid \mathcal{G}(u)=0\}$.
		
		For $\ker\mathcal{G}^{\prime}(u)\subset \Image\varphi^{\prime}(\omega)$, let $v\in V$ with $\mathcal{G}^{\prime}(u)v = 0$ and write $v=v_0+v_1\in V_0\oplus V_1$. Then 
		$0 = \frac{\partial\mathcal{G}}{\partial v_0}(u) v_0 + \frac{\partial\mathcal{G}}{\partial v_1}(u)v_1$. With $\varphi, \psi$ as in \Cref{thm:localgraph} and writing $u = \varphi(\omega)$ we conclude using \eqref{eq:psi'}
		\begin{align*}
		v_1 = - \left(\frac{\partial\mathcal{G}}{\partial v_1}(u)\right)^{-1} \frac{\partial \mathcal{G}}{\partial v_0}(u) v_0 = \psi^{\prime}(\omega) v_0.
		\end{align*}
		Consequently, $\varphi^{\prime}(\omega)v_0 = v_0 + \psi^{\prime}(\omega)v_0  =v_0+v_1=v$, thus $v\in \Image\varphi^{\prime}(\omega)$.
		
		To prove $\Image \varphi^{\prime}(\omega)\subset \mathcal{T}_u\mathcal{M}$, let $\omega\in \Omega_0$ with $\varphi(\omega)=u$ and let $y\in \Image \varphi^{\prime}(\omega)$. Then there exists $v\in V_0$ with $y=\varphi^{\prime}(\omega)v$, hence $y=\varphi^{\prime}(\omega)v = \dtzero \varphi(\omega+tv) \in \mathcal{T}_u\mathcal{M}$ since $\gamma(t)\defeq \varphi(\omega+tv)$ defines a curve in $\mathcal{M}$ with $\gamma(0)=u$ and $\gamma^{\prime}(0)=\varphi^{\prime}(\omega)v$.

		\item First, let $y\in \overline{\mathcal{T}_u\mathcal{M}}$ and $v_n\in \mathcal{T}_u\mathcal{M}$ with $v_n\to y$ in $Y$. Using the extension property (iv) in \Cref{thm:main}, we get $\overline{\mathcal{G}^{\prime}(u)}y = \lim_{n\to\infty} \mathcal{G}^{\prime}(u) v_n = 0$ by (i).
		
		Conversely, let $y\in Y$ such that $\overline{\mathcal{G}^{\prime}(u)} y=0$. Since $\frac{\partial \mathcal{G}}{\partial v_1}(u)\colon V_1\to \R^m$ is an isomorphism by the proof of \Cref{thm:localgraph}, we conclude that $\mathcal{G}^{\prime}(u)\colon V\to\R^m$ is surjective. As a consequence, $\R^m\cong {V}/{\ker\mathcal{G}^{\prime}(u)}$, so $\codim\ker\mathcal{G}^{\prime}(u)=m$ is finite. By \Cref{lem:fin_co_dim complemented}, there exists a closed subspace $W=W(u)$ of $V$ with $V=\ker\mathcal{G}^{\prime}(u)\oplus W$. By density of $V$ in $Y$, there exists a sequence $(v_n)\subset V$ with $v_n\to y$ in $Y$. Thus, we may write $v_n = v_n^0 + w_n$ with $v_n^0\in \ker\mathcal{G}^{\prime}(u)$ and $w_n\in W$. As a consequence, $\mathcal{G}^{\prime}(u)v_n = \mathcal{G}^{\prime}(u)w_n \to \overline{\mathcal{G}^{\prime}(u)}y = 0$ in $\R^m$. Since $\mathcal{G}^{\prime}(u)\vert_{W}\colon W\to\R^m$ is an isomorphism, we get $w_n\to 0$ in $W\subset V$, hence in $Y$. Thus, $y = \lim_{n\to\infty} v_n = \lim_{n\to\infty} v_n^0$ with $v_n^0\in \ker\mathcal{G}^{\prime}(u) = \mathcal{T}_u\mathcal{M}$ by (i).
		
		\item First, since $\frac{\partial\mathcal{G}}{\partial v_1}(u)\colon V_1\to\R^m$ is an isomorphism by the proof of \Cref{thm:localgraph}, the operator $\mathcal{G}^{\prime}(u)\colon V\to\R^m $ is surjective. Thus $\R^m \cong {V}/{\ker\mathcal{G}^{\prime}(u)}={V}/{\mathcal{T}_u\mathcal{M}}$ by (i), so $\codim(\mathcal{T}_u\mathcal{M},V)=m$.  Moreover, since $\mathcal{G}^{\prime}(u)\colon V\to\R^m$ is surjective, so is the extension $\overline{\mathcal{G}^{\prime}(u)}\colon Y\to\R^m$ and hence $\R^m \cong{Y}/{\ker\overline{\mathcal{G}^{\prime}(u)}} = {Y}/{\overline{\mathcal{T}_u\mathcal{M}}}$ by (ii), so $\codim(\overline{\mathcal{T}_u\mathcal{M}},Y)=m$.\qedhere
	\end{enumerate}
\end{proof}
\begin{rem}\label{rem:projTbaruM}
	In particular, \Cref{prop:propertiesTuM} and \Cref{lem:fin_co_dim complemented} imply that there exists a projection $P(\bar{u})\colon Y\to Y$ onto $\overline{\mathcal{T}_{\bar{u}}\mathcal{M}} =\overline{\ker\mathcal{G}^{\prime}(u)} \eqdef \overline{V_0}$. 
\end{rem}
As a next step, we investigate the properties of the chart $\varphi$ defined in \Cref{thm:localgraph} under the assumptions on $\mathcal{G}$ in \Cref{thm:main}.
\begin{prop}\label{prop:propertiesofchart}
	Suppose $\mathcal{G}$ satisfies assumptions (iv) and (vi) in \Cref{thm:main}. Using the notation of \Cref{thm:localgraph}, we have
	\begin{enumerate}[(i)]
		\item for all $\omega\in \Omega_0$, the operator $\psi^{\prime}(\omega)\colon V_0\to V$ defined in \eqref{eq:psi'} extends to an operator $ \overline{\psi^{\prime}(\omega)}\in \mathcal{L}(Y)$ such that $\sup_{\omega\in \Omega_0}\norm{\overline{\psi^{\prime}(\omega)}}{\mathcal{L}(Y)} <\infty$ and $\Image \overline{\psi^{\prime}(\omega)}\subset V_1$ is finite-dimensional, replacing $\Omega_0$ with a smaller neighborhood if necessary,
		\item for any $\omega\in \Omega_0$, the operator $\varphi^{\prime}(\omega)$ extends to $\overline{\varphi^{\prime}(\omega)} = \Id_{Y}+\overline{\psi^{\prime}(\omega)}\colon Y\to Y,$
		\item the map $\overline{\varphi^{\prime}} \colon \Omega_0\to \mathcal{L}(Y),\omega \mapsto\overline{\varphi^{\prime}(\omega)}$ is analytic,
		\item $\overline{\varphi^{\prime}(\bar{\omega})}y = y$ for all $y\in \overline{V_0}$, where $\bar{\omega}\in \Omega_0$ satisfies $\varphi(\bar{\omega})=\bar{u}$.
		\item For all $\omega\in \Omega_0$ and $y\in \overline{V_0}$, we have $\norm{\overline{\varphi^{\prime}(\omega)} y}{Y}\geq \frac{1}{2} \norm{y}{Y}$, passing to a smaller neighborhood $\Omega_0$ if necessary.
	\end{enumerate}
\end{prop}
\begin{proof}
	\begin{enumerate}[(i)]
		\item By assumption (iv) in \Cref{thm:main}, the operator $\mathcal{G}^{\prime}(\varphi(\omega))$ extends to an operator $\overline{\mathcal{G}^{\prime}(\varphi(\omega))}\colon Y \to \R^m$ and hence $\psi^{\prime}(\omega)$ extends to an operator $\overline{\psi^{\prime}(\omega)}\colon Y\to Y$ via \eqref{eq:psi'}. Since the image of $\big(\frac{\partial\mathcal{G}}{\partial v_1}(\varphi(\omega))\big)^{-1}\colon \R^m \to V_1$ is contained in $V_1$, we conclude that $\Image \overline{\psi^{\prime}(\omega)}\subset V_1\subset Y$ is finite-dimensional. For the norm estimate note that for any $y\in Y$, using $V\hookrightarrow Y$, we have
		\begin{align*}
		\Norm{\overline{\psi^{\prime}(\omega)}y}{Y} &\leq C \Norm{\big(\frac{\partial\mathcal{G}}{\partial v_1}(\varphi(\omega))\big)^{-1}}{\mathcal{L}(\R^m, V_1)} \Norm{\overline{\mathcal{G}^{\prime}(\varphi(\omega))}}{\mathcal{L}(Y,\R^m)}\Norm{y}{Y}\leq CC^{\prime} \Norm{y}{Y}
		\end{align*}
		for some $C, C^{\prime} >0$, passing to a smaller $\Omega_0$ if necessary, since $\varphi$ and $\mathcal{G}$ are analytic and so is the extension $\overline{\mathcal{G}^{\prime}}\colon U\to \mathcal{L}(Y,\R^m)$ by assumption (iv) in \Cref{thm:main}.
		\item This follows from (i) and \eqref{eq:varphi'}.
		\item The map $\Omega_0\ni \omega\mapsto\overline{\mathcal{G}^{\prime}(\varphi(\omega))}\in \mathcal{L}(Y,\R^m)$ is analytic using assumption (iv) in \Cref{thm:main} and the analyticity of $\varphi$. Therefore, using \eqref{eq:varphi'}, so is the extension $\Omega_0 \ni \omega\mapsto\overline{\varphi^{\prime}(\omega)}\in \mathcal{L}(Y)$ using \Cref{thm:companalytic} and \Cref{ex:multilinAnalytic}.
		\item By \Cref{prop:propertiesTuM} (ii) and \eqref{eq:psi'},  $\overline{\psi^{\prime}(\omega)}y=0$ for all $y\in \overline{V_0}$. This yields the claim.
		\item By (iv), $\norm{\overline{\varphi^{\prime}(\bar{\omega})}y}{Y} = y$ for all $y\in \overline{V_0}$. By (iii), passing to a smaller $\Omega_0$ if necessary, we can assume $\norm{\overline{\varphi^{\prime}({\bar\omega})}-\overline{\varphi^{\prime}({\omega})}}{\mathcal{L}(Y)} \leq \frac{1}{2}$ for all $w\in \Omega_0$. Then, for any $y\in \overline{V_0}$ we can estimate 
		$\norm{\overline{\varphi^{\prime}({\omega})} y}{Y} \geq \norm{\overline{\varphi^{\prime}(\bar{\omega})}y}{Y} - \norm{\overline{\varphi^{\prime}({\bar\omega})}y-\overline{\varphi^{\prime}({\omega})}y}{Y}\geq \left(1-\frac{1}{2}\right)\norm{y}{Y}$.
		\qedhere
	\end{enumerate}
\end{proof}

\section{Proof of the \texorpdfstring{{\L}ojasiewicz}{Lojasiewicz}--Simon gradient inequality}\label{sec:ProofMain}
In this section, we will establish the {\L}ojasiewicz--Simon gradient inequality for the energy $\mathcal{E}$ composed with the chart we constructed in \Cref{sec:chartgraph} and use this to prove our main theorem. 
\begin{thm}\label{thm:propertiesofF}
	Suppose $\mathcal{E}$, $\mathcal{G}$ and $\bar{u}\in U$ satisfy assumptions (i), (ii), (iv) and (vi) of \Cref{thm:main}. Let $\varphi$ be the chart centered at $\bar{u}$ defined in \Cref{thm:localgraph}. Define $\mathcal{F}\colon \Omega_0\to\R$, $\mathcal{F}(\omega)\defeq (\mathcal{E}\circ\varphi)(\omega)$. Then 
	\begin{enumerate}[(i)]
		\item $\mathcal{F}$ is analytic,
		\item for $\omega\in \Omega_0$, $\mathcal{F}^{\prime}(\omega)\in \overline{V_0}^{\ast}$ via $\mathcal{F}^{\prime}(\omega) = P(\bar{u})^{\ast}\circ \overline{\varphi^{\prime}(\omega)}^{\ast}\mathcal{E}^{\prime}(\varphi(\omega))$, where $P(\bar{u})\in \mathcal{L}(Y)$ is the projection onto $\overline{V_0} = \overline{\mathcal{T}_{\bar{u}}\mathcal{M}}$ from \Cref{rem:projTbaruM},
		\item the map $\Omega_0\ni \omega\mapsto \mathcal{F}^{\prime}(\omega)\in \overline{V_0}^{\ast}$ is analytic.
	\end{enumerate}
\end{thm}
\begin{proof}
	\begin{enumerate}[(i)]
		\item This follows from \Cref{thm:companalytic} since $\mathcal{E}$ and $\varphi$ are analytic.
		\item Let $\omega\in \Omega_0, v\in V_0$ and $P(\bar{u})\in\mathcal{L}(Y)$ be as in \Cref{rem:projTbaruM}. We compute
		\begin{align*}
		\mathcal{F}^{\prime}(\omega)v &= \mathcal{E}^{\prime}(\varphi(\omega))\circ\varphi^{\prime}(\omega)v = \mathcal{E}^{\prime}(\varphi(\omega))\circ\overline{\varphi^{\prime}(\omega)}v \\
		&= \overline{\varphi^{\prime}(\omega)}^{\ast}\circ\mathcal{E}^{\prime}(\varphi(\omega))(P(\bar{u})v) = P(\bar{u})^{\ast} \circ \overline{\varphi^{\prime}(\omega)}^{\ast} \circ \mathcal{E}^{\prime}(\varphi(\omega)) v
		\end{align*} using that $P(\bar{u})v=v$ since $v \in V_0 \subset \overline{V_0}$. Since $P(\bar{u})$ projects onto $\overline{V_0}^{\ast},$ (ii) follows.
		\item The chart $\varphi\colon\Omega_0\to U$ is analytic and so is $\mathcal{E}^{\prime}\colon U \to Y^{\ast}$ by assumption (ii) in \Cref{thm:main}. Thus, so is their composition $\omega\mapsto \mathcal{E}^{\prime}(\varphi(\omega))\colon \Omega_0\to Y^{\ast}$ by \Cref{thm:companalytic}. By \Cref{prop:propertiesofchart}, the extension $\overline{\varphi^{\prime}}\colon \Omega_0 \to \mathcal{L}(Y)$ is analytic, and so is taking the adjoint $T\mapsto T^{\ast}\colon \mathcal{L}(Y)\to \mathcal{L}(Y^{\ast})$ by \Cref{ex:multilinAnalytic}, since it is linear and bounded. Similarly, the evaluation map $\mathcal{L}(Y^{\ast})\times Y^{\ast}\to Y^{\ast}, (T, y^{\ast})\mapsto Ty^{\ast}$ is analytic. Therefore, $\Omega_0\to \overline{V_0}^{\ast}, \omega \mapsto \mathcal{F}^{\prime}(\omega) = P(\bar{u})\circ \overline{\varphi^{\prime}(\omega)}^{\ast}\circ \mathcal{E}^{\prime}(\varphi(\omega))$ is analytic, since the projection $P(\bar{u})\colon Y\to \overline{V_0}^{\ast}$ is analytic. \qedhere
	\end{enumerate}
\end{proof}
The following lemma justifies our approach to study $\mathcal{F}$ in order to prove \Cref{thm:main}.
\begin{lem}\label{lem:FLojaIffE_MLoja}
	Suppose $\mathcal{E}$ and $\mathcal{G}$ are analytic and satisfy assumptions (i), (ii), (iv) and (vi) in \Cref{thm:main}. Let $\bar{u}\in U$ and let $\varphi$ be the chart centered at $\bar{u}$ defined in \Cref{thm:localgraph}. Let $\bar{\omega}\in \Omega_0$ such that $\varphi(\bar{\omega}) = \bar{u}$ and $\mathcal{F}= \mathcal{E}\circ \varphi$ as in \Cref{thm:propertiesofF}. Then the following are equivalent.
	\begin{enumerate}[(i)]
		\item $\mathcal{F}$ satisfies a {\L}ojasiewicz--Simon gradient inequality at $\bar{\omega}$, i.e. there exist $C, \sigma^{\prime}>0$ and $\theta\in (0, \frac{1}{2}]$ such that
		\begin{align*}
		|\mathcal{F}(\omega)-\mathcal{F}(\bar{\omega})|^{1-\theta}\leq C \Norm{\mathcal{F}^{\prime}(\omega)}{\overline{V_0}^{\ast}} \text{ for all }\Norm{\omega-\bar{\omega}}{V}\leq \sigma^{\prime}.
		\end{align*}
		\item $\mathcal{E}\vert_{\mathcal{M}}$ satisfies a refined {\L}ojasiewicz--Simon gradient inequality \eqref{eq:LSnonlin} near $\bar{u}$.
	\end{enumerate}
\end{lem}

\begin{proof} Suppose (i) holds. Let $u\in \mathcal{M}$ with $\Norm{u-\bar{u}}{V}\leq \sigma$. For $\sigma>0$ small enough, we can assume $u\in \Omega$, $u=\varphi(\omega)$ for a unique $\omega\in \Omega_0$ and $\Norm{\omega-\bar{\omega}}{V}\leq \sigma^{\prime}$ by continuity. Then by (i), we have
	\begin{align}\label{eq:Lojacomposed0}
	|\mathcal{E}(u)-\mathcal{E}(\bar{u})\vert^{1-\theta} &= \vert \mathcal{F}(\omega) - \mathcal{F}(\bar{\omega})\vert^{1-\theta}\leq C \Norm{\mathcal{F}^{\prime}(\omega)}{\overline{V_0}^{\ast}}.
	\end{align}
	Now, by \Cref{prop:propertiesTuM}, we have $\mathcal{T}_{\varphi(\omega)}\mathcal{M} = \Image\varphi^{\prime}(\omega)$ and thus $\varphi^{\prime}(\omega)v\in \mathcal{T}_{\varphi(\omega)}\mathcal{M}$ for $v\in V_0$.
	By continuity of the extension (see \Cref{prop:propertiesofchart} (ii)), we get $\overline{\varphi^{\prime}(\omega)}y\in \overline{\mathcal{T}_{\varphi(\omega)}\mathcal{M}}$ for $y\in \overline{V_0}$. Using $P(\bar{u})y=y$ for $y\in \overline{V_0}$ and \Cref{thm:propertiesofF} (ii), we compute
	\begin{align}\label{eq:nablaFEstimate0}
	\Norm{\mathcal{F}^{\prime}(\omega)}{\overline{V_0}^{\ast}} &= \sup_{0\neq y\in \overline{V_0}} \frac{P(\bar{u})^{\ast}\circ \overline{\varphi^{\prime}({\omega})}^{\ast} \circ \mathcal{E}^{\prime}(u)y}{\Norm{y}{Y}} = \sup_{0\neq y\in \overline{V_0}} \frac{\overline{\varphi^{\prime}({\omega})}^{\ast} \circ \mathcal{E}^{\prime}({u})(P(\bar{u})
		y)}{\Norm{y}{Y}}\nonumber\\
	&= \sup_{0\neq y\in \overline{V_0}} \frac{\mathcal{E}^{\prime}(u)(\overline{\varphi^{\prime}(\omega)}y)}{\Norm{y}{Y}} \leq  \sup_{0\neq y\in \overline{V_0}} \frac{\Norm{\mathcal{E}^{\prime}({u})}{\overline{\mathcal{T}_{u}\mathcal{M}}^{\ast}} \norm{\overline{\varphi^{\prime}(\omega)}y}{\overline{\mathcal{T}_{u}\mathcal{M}}}}{\Norm{y}{Y}}\nonumber \\
	&\leq \Norm{\mathcal{E}^{\prime}({u})}{\overline{\mathcal{T}_{u}\mathcal{M}}^{\ast}} \sup_{0\neq y\in \overline{V_0}} \frac{ \norm{\overline{\varphi^{\prime}(\omega)}y}{Y}}{\Norm{y}{Y}} \leq \Norm{\mathcal{E}^{\prime}({u})}{\overline{\mathcal{T}_{u}\mathcal{M}}^{\ast}} \sup_{\omega\in \Omega_0} \norm{\overline{\varphi^{\prime}(\omega)}}{\mathcal{L}(Y)}.
	\end{align}
	Now, reducing $\sigma, \sigma^{\prime}>0$ if necessary, and using \Cref{prop:propertiesofchart} (i), we may assume that $\sup_{v\in \Omega_0} \norm{\overline{\varphi^{\prime}(v)}}{\mathcal{L}(Y)}\leq 1+\sup_{v\in \Omega_0} \norm{\overline{\psi^{\prime}(v)}}{\mathcal{L}(Y)}  \leq C^{\prime}<\infty$. Hence, using \eqref{eq:Lojacomposed0} and \eqref{eq:nablaFEstimate0}, we conclude $\abs{\mathcal{E}(u)-\mathcal{E}(\bar{u})}^{1-\theta}\leq CC^{\prime} \Norm{\mathcal{E}^{\prime}(u)}{\overline{\mathcal{T}_u\mathcal{M}}^{\ast}}$.
	
	Conversely, suppose $\mathcal{E}\vert_{\mathcal{M}}$ satisfies \eqref{eq:LSnonlin} for some $C, \sigma>0$ and $\theta\in (0, \frac{1}{2}]$. Let $\omega\in \Omega_{0}$ with $\norm{\omega-\bar{\omega}}{V}\leq \sigma^{\prime}$. Define $u\defeq\varphi(\omega)$ and let $\sigma^{\prime}>0$ be small enough such that $\norm{u-\bar{u}}{V}\leq \sigma$. Then we have
	\begin{align}\label{eq:LS_F<E'}
	\abs{\mathcal{F}(\omega)-\mathcal{F}(\bar{\omega})}^{1-\theta} = \abs{\mathcal{E}(u)-\mathcal{E}(\bar{u})}^{1-\theta} \leq C\norm{\mathcal{E}^{\prime}(u)}{\overline{\mathcal{T}_u\mathcal{M}}^{\ast}}.
	\end{align}
	Now, fix $0\neq w\in {\mathcal{T}_u\mathcal{M}}$. Then by \Cref{prop:propertiesTuM}, $w=\varphi^{\prime}(\omega)v$ for some $0\neq v\in V_0$. We have
	\begin{align*}
	\mathcal{E}^{\prime}(u)w = \mathcal{E}^{\prime}(u)(\varphi^{\prime}(\omega)v) = \mathcal{E}^{\prime}(u)(\overline{\varphi^{\prime}(\omega)}v)  = \overline{\varphi^{\prime}({\omega})}^{\ast}\circ \mathcal{E}^{\prime}(u)(P(\bar{u})v) = \mathcal{F}^{\prime}(\omega)v
	\end{align*}
	using $P(\bar{u}) v = v$ since $v\in \overline{V_0}$ and \Cref{thm:propertiesofF} (ii). Thus, we find
	\begin{align}\label{eq:E'<F'estimate}
	\frac{\mathcal{E}^{\prime}(u)w}{\norm{w}{Y}} = \frac{\mathcal{F}^{\prime}(\omega)v}{\norm{\overline{\varphi^{\prime}(\omega)}v}{Y}} \leq 2\frac{\mathcal{F}^{\prime}(\omega)v}{\norm{v}{Y}} \leq 2\sup_{0\neq v\in V_0}\frac{\mathcal{F}^{\prime}(\omega)v}{\norm{v}{Y}} \leq 2\sup_{0\neq y\in \overline{V_0}}\frac{\mathcal{F}^{\prime}(\omega)y}{\norm{y}{Y}} = 2\Norm{\mathcal{F}^{\prime}(\omega)}{\overline{V_0}^{\ast}},
	\end{align}
	using \Cref{prop:propertiesofchart} (v) and reducing $\sigma^{\prime}>0$ if necessary. Since $\mathcal{E}^{\prime}(u)\in Y^{\ast}$ by assumption,  we may conclude that \eqref{eq:E'<F'estimate} remains valid if $0\neq w\in \overline{\mathcal{T}_u\mathcal{M}}$ by continuity. Combining this with \eqref{eq:LS_F<E'}, the claim follows.
\end{proof}

\begin{thm}\label{thm:DF'Fredholm}
	Suppose $\mathcal{E}, \mathcal{G}$ and $\bar{u}\in U$ satisfy the assumptions of \Cref{thm:main} above and let $\mathcal{F}$ be as in \Cref{thm:propertiesofF}. Then, for $\bar{\omega}\in\Omega_0$ with $\varphi(\bar{\omega})=\bar{u}$, the operator $\mathcal{F}^{\prime\prime}(\bar{\omega})\colon V_0\to \overline{V_0}^{\ast}$ is Fredholm of index zero.
\end{thm}
\begin{proof}
	Let $\varphi(\bar{\omega})=\bar{u}$ and $v\in V_0$. By \Cref{prop:propertiesofchart} (iv), we have $\overline{\varphi^{\prime}(\omega)}v=v$. We use the chain rule, the analyticity of $\varphi, \psi, \overline{\varphi^{\prime}}, \overline{\psi^{\prime}}$ and $\mathcal{E}^{\prime}\colon U\to Y^{\ast}$, and the analyticity and bilinearity of the evaluation map $\mathcal{L}(Y^{\ast})\times Y^{\ast}\to Y^{\ast}$ to compute
	\begin{align}
	\label{eq:DF'Formula}
	\mathcal{F}^{\prime\prime}(\bar{\omega}) v &= \dtzero~ \mathcal{F}^{\prime}(\bar{\omega}+tv) = P(\bar{u})^{\ast}\circ\dtzero\Big( \overline{\varphi^{\prime}(\bar{\omega}+tv)}^{\ast} \mathcal{E}^{\prime}(\varphi(\bar{\omega}+tv))\Big) \nonumber \\
	&= P(\bar{u})^{\ast}\circ \dtzero\Big(\mathcal{E}^{\prime}(\varphi(\bar{\omega}+tv))+\overline{ \psi^{\prime}(\bar{\omega}+tv)}^{\ast} \mathcal{E}^{\prime}(\varphi(\bar{\omega}+tv)) \Big) \nonumber\\
	&= P(\bar{u})^{\ast}\circ \mathcal{E}^{\prime\prime}(\bar{u})v+ P(\bar{u})^{\ast}\Bigg( \left(\dtzero~ \overline{\psi^{\prime}(\bar{\omega}+tv)}\right)^{\ast} \mathcal{E}^{\prime}(\bar{u}) + \overline{\psi^{\prime}(\bar{\omega})}^{\ast} \mathcal{E}^{\prime\prime}(\bar{u})v\Bigg)\nonumber\\
	&\eqdef P(\bar{u})^{\ast}\circ \mathcal{E}^{\prime\prime}(\bar{u})v +  Ev,
	\end{align}
	using \Cref{thm:propertiesofF} (ii), \Cref{prop:propertiesofchart} and the fact that $\overline{\varphi^{\prime}}^{\ast}=\Id_{Y^{\ast}}+\overline{\psi^{\prime}}^{\ast}$ by \eqref{eq:varphi'}.
	We will now show that $E\colon V_0\to \overline{V_0}^{\ast}$ is compact. First, using \eqref{eq:psi'}, we compute
	\begin{align*}
	\dtzero~ \overline{\psi^{\prime}(\bar{\omega}+tv)} &= \dtzero\Bigg(-\left(\frac{\partial\mathcal{G}}{\partial v_1}(\varphi(\bar{\omega}+tv))\right)^{-1} \circ \Big(\overline{\mathcal{G}^{\prime}(\varphi(\bar{\omega}+tv))}\Big)\Bigg)\\
	&= -\left(\dtzero~ \left(\frac{\partial\mathcal{G}}{\partial v_1}(\varphi(\bar{\omega}+tv))\right)^{-1}\right)\circ \overline{\mathcal{G}^{\prime}(\bar{u})} \\
	&\quad - \left( \frac{\partial\mathcal{G}}{\partial v_1}(\bar{u})\right)^{-1} \circ \dtzero~ \overline{\mathcal{G}^{\prime}(\varphi(\bar{\omega}+tv))} \\
	&= -\left(\dtzero~ \left(\frac{\partial\mathcal{G}}{\partial v_1}(\varphi(\bar{\omega}+tv))\right)^{-1}\right)\circ \overline{\mathcal{G}^{\prime}(\bar{u})} \\
	&\quad - \left( \frac{\partial\mathcal{G}}{\partial v_1}(\bar{u})\right)^{-1} \circ (\overline{\mathcal{G}^{\prime}})^{\prime}(\bar{u}) v \eqdef -Rv - Sv,
	\end{align*}
	where $R, S \in \mathcal{L}(V_0, \mathcal{L}(Y))$. We conclude that 
	\begin{align}\label{eq:DF'compactPerturbation1}
	\left(\dtzero~ \overline{\psi^{\prime}(\bar{\omega}+tv)}\right)^{\ast}\mathcal{E}^{\prime}(\bar{u}) & = -(Rv)^{\ast}\mathcal{E}^{\prime}(\bar{u}) - (Sv)^{\ast}\mathcal{E}^{\prime}(\bar{u}).
	\end{align}
	Recall that for $u\in \Omega=\varphi(\Omega_0)$, $\left(\frac{\partial\mathcal{G}}{\partial v_1}(u)\right)^{-1}\colon \R^m\to V_1\hookrightarrow Y$. The first part of \eqref{eq:DF'compactPerturbation1} is
	\begin{align}\label{eq:DF'compactPerturbation2}
	(Rv)^{\ast}\mathcal{E}^{\prime}(\bar{u}) &= \left(\overline{\mathcal{G}^{\prime}(\bar{u})}\right)^{\ast}\circ \left(\dtzero~ \left(\frac{\partial\mathcal{G}}{\partial v_1}(\varphi(\bar{\omega}+tv))\right)^{-1}\right)^{\ast} \mathcal{E}^{\prime}(\bar{u}),
	\end{align}
	so the image of $v\mapsto (Rv)^{\ast}\mathcal{E}^{\prime}(\bar{u})$ is contained in $\Image (\overline{\mathcal{G}^{\prime}(\bar{u}})^{\ast}\subset Y^{\ast}$, which is finite-dimensional since $\overline{\mathcal{G}^{\prime}(\bar{u})}\colon Y\to\R^m$ has finite rank.
	
	Furthermore, with $\eta \defeq \left( \left( \frac{\partial\mathcal{G}}{\partial v_1}(\bar{u})\right)^{-1}\right)^{\ast}\mathcal{E}^{\prime}(\bar{u})\in \R^m$
	we have
	\begin{align}\label{eq:DF'compactPerturbation3}
	(Sv)^{\ast}\mathcal{E}^{\prime}(\bar{u})  &= \left( (\overline{\mathcal{G}^{\prime}})^{\prime}(\bar{u}) v\right)^{\ast}\circ \left( \left( \frac{\partial\mathcal{G}}{\partial v_1}(\bar{u})\right)^{-1}\right)^{\ast}\mathcal{E}^{\prime}(\bar{u})=  \left( (\overline{\mathcal{G}^{\prime}})^{\prime}(\bar{u}) v\right)^{\ast} \eta.
	\end{align}
	
	We will now show that the operator $v\mapsto(Sv)^{\ast}\mathcal{E}^{\prime}(\bar{u})\colon V_0\to Y^{\ast}$ is compact. Let $v_n\in V_0$ for $n\in\N$ with $\Norm{v_n}{V}\leq 1$. By assumption (v) in \Cref{thm:main}, passing to a subsequence, we can assume $(\overline{\mathcal{G}^{\prime}})^{\prime}(\bar{u})v_n\to A$ in $\mathcal{L}(Y,\R^m)$. Since taking the adjoint is continuous, this yields $\left((\overline{\mathcal{G}^{\prime}})^{\prime}(\bar{u})v_n\right)^{\ast}\to A^{\ast}$ in $\mathcal{L}(\R^m, Y^{\ast})$. But this clearly implies $\left(\overline{\mathcal{G}^{\prime}})^{\prime}(\bar{u})v_n\right)^{\ast} \eta\to A^{\ast}\eta $ in $Y^{\ast}$.
	
	By the previous arguments, together with \eqref{eq:DF'compactPerturbation1}, \eqref{eq:DF'compactPerturbation2} and \eqref{eq:DF'compactPerturbation3}, we conclude that the linear operator $V_0\to Y^{\ast}, v\mapsto \left(\dtzero~ \overline{\psi^{\prime}(\bar{\omega}+tv)}\right)^{\ast}\mathcal{E}^{\prime}(\bar{u})$ is compact.
	
	Clearly, since $\overline{\psi^{\prime}(\bar{\omega})}$ has finite rank, the image of $v\mapsto \overline{\psi^{\prime}(\bar{\omega})}^{\ast}\mathcal{E}^{\prime\prime}(\bar{u})v$ is contained in $\Image \left(\overline{\psi^{\prime}(\bar{\omega})}\right)^{\ast}$ and thus finite-dimensional. As a consequence, 
	\begin{align*}
	E\colon V_0\to \overline{V_0}^{\ast}, Ev = P(\bar{u})^{\ast}\left(-(Rv)^{\ast}\mathcal{E}^{\prime}(\bar{u}) - (Sv)^{\ast}\mathcal{E}^{\prime}(\bar{u}) + \overline{\psi^{\prime}(\bar{\omega})}^{\ast}\mathcal{E}^{\prime\prime}(\bar{u})v\right)
	\end{align*}
	is a compact operator. We will now show that $\mathcal{F}^{\prime\prime}(\bar{\omega})\colon V_0\to \overline{V_0}^{\ast}$ is Fredholm with $\ind{\mathcal{F}^{\prime\prime}(\overline{\omega})}=0$. By \eqref{eq:DF'Formula} and \Cref{prop:Fredholmproperties}, it is enough to show that $T\colon V\mapsto \overline{V_0}^{\ast},$ $v\mapsto P(\bar{u})^{\ast}\circ\mathcal{E}^{\prime\prime}(\bar{u})v$ is Fredholm of index zero.
	
	Note that $T = P(\bar{u})^{\ast}\circ \mathcal{E}^{\prime\prime}(\bar{u})\circ \iota$. Here, $\iota\colon V_0\hookrightarrow V$ is the inclusion, which
	is Fredholm since $\ker\iota=\{0\}$ and $\operatorname{codim}(\Image\iota, V) = \operatorname{codim}(V_0,V)= \codim(\mathcal{T}_{\bar{u}}\mathcal{M}, V)=
	m$ by \Cref{prop:propertiesTuM}. Moreover, recall from \Cref{rem:projTbaruM} that $P(\bar{u})\colon Y\to Y$ is the projection onto $\overline{V_0} = \overline{\mathcal{T}_{\bar{u}}\mathcal{M}}$ with $\codim(\overline{V_0},Y)=m$ by \Cref{prop:propertiesTuM}. Thus $Y = \overline{V_0}\oplus Z$ with $\dim Z=m$. Then $\ker P(\bar{u})^{\ast} \cong Z^{\ast}$, so $\dim \ker P(\bar{u})^{\ast} = \dim Z^{\ast} = m$. Clearly, $\codim(\Image P(\bar{u})^{\ast},\overline{V_0}^{\ast})=0$.
	
	Therefore, by \Cref{thm:FredholmIndexFormula}, the composition $T = P(\bar{u})^{\ast}\circ \mathcal{E}^{\prime\prime}(\bar{u}) \circ \iota$ is Fredholm with $\ind{T} = \ind{P(\bar{u})}^{\ast}+\ind{\mathcal{E}^{\prime\prime}(\bar{u})}+\ind{\iota} = m + \ind{\mathcal{E}^{\prime\prime}(\bar{u})} - m = 0$.
\end{proof}
Now, it is not difficult to see that $\mathcal{F}$ satisfies a {\L}ojasiewicz--Simon gradient inequality at a critical point by \Cref{thm:FredholmLoja}.
\begin{thm}\label{thm:lojacomposed}
	Suppose $\mathcal{E}, \mathcal{G}$ and $\bar{u}\in \mathcal{M}$ satisfy the assumptions of \Cref{thm:main}. Let $\varphi$ be the chart constructed in \Cref{thm:localgraph} with $\bar{v}\in \Omega_0$ such that $\varphi(\bar{\omega})=\bar{u}$. If $\mathcal{F}^{\prime}(\bar{\omega})=0$, then there exist $C,\sigma^{\prime}>0$ and $\theta\in(0,\frac{1}{2}]$ such that
	\begin{align*}
	|\mathcal{F}(\omega)-\mathcal{F}(\bar{\omega})|^{1-\theta}\leq C \Norm{\mathcal{F}^{\prime}(\omega)}{\overline{V_0}^{\ast}} \text{ for all }\Norm{\omega-\bar{\omega}}{V}\leq \sigma^{\prime}.
	\end{align*}
\end{thm}
\begin{proof}
	We verify that the assumptions in \Cref{thm:FredholmLoja} are satisfied for $V=V_0$, $U=\Omega_0$ $Z=\overline{V_0}$, $\varphi=\bar{u}$ and $\mathcal{E}=\mathcal{F}$. Density of $V_0\subset \overline{V_0}$ is trivial.
	
	Assumption (ii) in \Cref{thm:FredholmLoja} is satisfied by \Cref{thm:propertiesofF} (iii). Assumption (iii), i.e. the Fredholm property of $\mathcal{F}^{\prime\prime}(\bar{\omega})\colon V_0 \to \overline{V_0}^{\ast}$, holds by \Cref{thm:DF'Fredholm}. Hence, $\mathcal{F}$ satisfies a {\L}ojasiewicz--Simon gradient inequality in a neighborhood of $\bar{\omega}$ by \Cref{thm:FredholmLoja}.
\end{proof}
We are finally able to prove our main result.
\begin{proof}[Proof of {\Cref{thm:main}}]
	Suppose $\mathcal{E}, \mathcal{G}$ and $\bar{u}\in \mathcal{M} = \{u\in \mathcal{M}\mid \mathcal{G}(u)=0\}$ satisfy the assumptions of \Cref{thm:main}. Suppose $\bar{u}$ is a constraint critical point in the sense of \Cref{prop:CharaConstCritPoint}. By \Cref{thm:localgraph}, \Cref{rem:Mmanifold} and \Cref{prop:propertiesTuM}, $\mathcal{M}$ is locally a manifold near $\bar{u}$ with codimension $m$. Let $\varphi\colon\Omega_0\to \Omega\cap \mathcal{M}$ be the chart from \Cref{thm:localgraph} centered at $\bar{u}$ with $\varphi(\bar{\omega})=\bar{u}$. Recall from \Cref{prop:CharaConstCritPoint} that $\mathcal{E}^{\prime}(\bar{u})v=0$ for all $v\in \mathcal{T}_{\bar{u}}\mathcal{M} = \Image\varphi^{\prime}(\bar{\omega})$ by \Cref{prop:propertiesTuM} (i).
	Then, for any $v\in V_0$, using \Cref{thm:propertiesofF} and $P(\bar{u})v = v$ we have 
	\begin{align*}
	\mathcal{F}^{\prime}(\bar{\omega}) v &= P(\bar{u})^{\ast}\circ \overline{\varphi^{\prime}(\bar{\omega})}^{\ast} \circ \mathcal{E}^{\prime}(\bar{u})v = \overline{\varphi^{\prime}(\bar{\omega})}^{\ast}\circ \mathcal{E}^{\prime}(\bar{u}) P(\bar{u}) v  \\
	&= \overline{\varphi^{\prime}(\bar{\omega})}^{\ast}\circ \mathcal{E}^{\prime}(\bar{u})v= \mathcal{E}^{\prime}(\bar{u}) \varphi^{\prime}(\bar{\omega}) v=0.
	\end{align*}
	Hence, by \Cref{thm:lojacomposed}, there exist $C, \sigma^{\prime}>0$ and $\theta\in(0,\frac{1}{2}]$ such that $\mathcal{F}$ satisfies a {\L}ojasiewicz--Simon gradient inequality at $\bar{\omega}$. By \Cref{lem:FLojaIffE_MLoja} the claim follows.
\end{proof}
\section{The Hilbert space framework}\label{sec:HilbertSpace}
In the setting where $Y=Y^{\ast}=H$ is a Hilbert space, the assumptions in \Cref{thm:main} can be characterized in a simpler way in terms of the $H$-gradients.
\begin{defnlem}\label{def:Hgrad}
	Let $V$ be a Banach space and let $(H, \langle\cdot, \cdot\rangle)$ be a Hilbert space such that $V\hookrightarrow H$ densely, so $H\hookrightarrow V^{\ast}$. Suppose $U\subset V$ is an open set and $\mathcal{E}\in \mathcal{C}^1(U;\R)$. If $\mathcal{E}^{\prime}(u) \in H$ under the identification of $H$ with its image in $V^{\ast}$, we say that $\mathcal{E}$ possesses an \emph{$H$-gradient at $u\in U$} and we write $\nabla\mathcal{E}(u)\defeq \mathcal{E}^{\prime}(u)\in H$. This means precisely that
	\begin{align}\label{eq:Hgrad}
	\mathcal{E}^{\prime}(u) v = \langle \nabla\mathcal{E}(u), v\rangle \text{ for all } v\in V,
	\end{align}
	i.e. $\mathcal{E}^{\prime}(u)\in V^{\ast} = \mathcal{L}(V,\R)$ extends to $\overline{\mathcal{E}^{\prime}(u)}\in \mathcal{L}(H,\R)$ via  \eqref{eq:Hgrad}. Thus, $\overline{\mathcal{E}^{\prime}(u)} = \nabla\mathcal{E}(u)$ under the isomorphism $H\cong H^{\ast}$ given by the Riesz--Fréchet Theorem.
\end{defnlem}

\begin{cor}\label{cor:main hilbert}
	Let $V$ be a Hilbert space, $U\subset V$ be an open set, $m\in\N$ and let $\mathcal{E}\in\mathcal{C}^{\omega}(U;\R), \mathcal{G}\in \mathcal{C}^{\omega}(U;\R^m)$. Let $\bar{u}\in U$ and suppose that
	\begin{enumerate}[(i)]
		\item there exists a Hilbert space $(H, \langle\cdot, \cdot\rangle)$ with $V\hookrightarrow H$ densely, 
		\item $\mathcal{E}$ possesses an $H$-gradient $\nabla\mathcal{E}(u)$ at each $u\in U$ and the map $u\mapsto \nabla\mathcal{E}(u)\colon U\to H$ is analytic,
		\item the second derivative $\mathcal{E}^{\prime\prime}(\bar{u}) =(\nabla\mathcal{E})^{\prime}(\bar{u}) \colon V\to H$ is Fredholm of index zero,\footnote{The equation $\mathcal{E}^{\prime\prime}(\bar{u}) = (\nabla\mathcal{E})^{\prime}(u)$ has to be understood in the sense of the identification $\mathcal{E}^{\prime}(u) =  \nabla \mathcal{E}(u)$, cf. \Cref{def:Hgrad}.}
		\item for any $u\in U$, the components $\mathcal{G}_k\colon U\to\R$ of $\mathcal{G}$ possess $H$-gradients $\nabla\mathcal{G}_k$ such that $U\ni u\mapsto \nabla\mathcal{G}_k(u)\in H$ is analytic for all $k=1, \dots, m$,
		\item the Fréchet derivatives $(\nabla\mathcal{G}_k)^{\prime}(\bar{u})\colon V\to H$ are compact for all $k=1,\dots, m$,
		\item $\mathcal{G}(\bar{u})=0$ and the $H$-gradients $\nabla\mathcal{G}_1(\bar{u}), \dots, \nabla\mathcal{G}_m(\bar{u})$ are linearly independent.
	\end{enumerate}
	Then, $\mathcal{M} \defeq \{ u\in U\mid \mathcal{G}(u)=0\}$ is locally an analytic submanifold of $V$ of codimension $m$ near $\bar{u}$. 
	
	If $\bar{u}$ is a critical point of $\mathcal{E}\vert_{\mathcal{M}}$, then the restriction satisfies a refined {\L}ojasiewicz--Simon gradient inequality at $\bar{u}$, i.e. there exist $C, \sigma>0$ and $\theta\in (0,\frac{1}{2}]$ such that for any $u\in \mathcal{M}$ with $\Norm{u-\bar{u}}{V}\leq \sigma$, we have
	\begin{align}\label{eq:LSHilbert}
	\abs{\mathcal{E}(u)-\mathcal{E}(\bar{u})}^{1-\theta}\leq C \Norm{P(u)\nabla\mathcal{E}(u)}{H},
	\end{align}
	where $P(u)\colon H\to H$ is the orthogonal projection onto $\overline{\mathcal{T}_u\mathcal{M}} \defeq\overline{\mathcal{T}_u\mathcal{M}} ^{\Norm{\cdot}{H}}$.
\end{cor}
%The proof of \Cref{cor:main hilbert} is given in \Cref{subsec:ProofHilbert}.
\begin{rem}
	Requiring $V$ to be a Hilbert space in \Cref{cor:main hilbert} is no additional assumption. Indeed, if hypothesis (iii) in \Cref{cor:main hilbert} is satisfied for $V$ merely a Banach space, $\mathcal{E}^{\prime\prime}(\bar{u})\colon V\to H$ is a compact perturbation of an isomorphism by \cite[Theorem 7.10]{AppellVaeth05}. In particular, $V$ and $H$ are isomorphic, so $V$ has to be a Hilbert space.
\end{rem}

\begin{rem}\label{rem:HilbertSpace m=1}
	In the case $m=1$ in \Cref{cor:main hilbert}, the projection $P(u)\in \mathcal{L}(H)$ onto $\overline{\mathcal{T}_u\mathcal{M}}$ is given  by
	$P(u) y = y - \frac{\langle\nabla\mathcal{G}(u), y\rangle}{\Norm{\nabla\mathcal{G}(u)}{H}^2}\nabla\mathcal{G}(u)$, where $\nabla\mathcal{G}\defeq\nabla\mathcal{G}_1$. This yields 
	\begin{align}\label{eq:LojaHilbertLambda}
	\abs{\mathcal{E}(u)-\mathcal{E}(\bar{u})}^{1-\theta}\leq  C \Big\Vert\nabla\mathcal{E}(u) - \frac{\langle\nabla\mathcal{G}(u), \nabla\mathcal{E}(u)\rangle}{\Norm{\nabla\mathcal{G}(u)}{H}^2} \nabla\mathcal{G}(u)\Big\Vert_H.
	\end{align}
	The scalar $\lambda(u) \defeq \frac{\langle\nabla\mathcal{G}(u), \nabla\mathcal{E}(u)\rangle}{\Norm{\nabla\mathcal{G}(u)}{}^{2}}$ is often referred to as the \emph{Langrange multiplier}, since if the right hand side of \eqref{eq:LojaHilbertLambda} is zero, $\lambda(u)$ is exactly the Lagrange multiplier for the function $\mathcal{E}$ subject to the constraint $\mathcal{G}(u)=0$ (cf. \cite[Chapter 2]{GiaquintaHildebrandt}).
\end{rem}
%\subsection{Proof of \texorpdfstring{\Cref{cor:main hilbert}}{Corollary 1.7}}\label{subsec:ProofHilbert}
The following shows that assumption (vi) in \Cref{cor:main hilbert} is just the equivalent formulation of hypothesis (vi) in \Cref{thm:main}.
\begin{lem}\label{lem:G'surjective}
	Let $V$ be a Banach space and let $(H, \langle\cdot, \cdot\rangle)$ be a Hilbert space such that $V\hookrightarrow H$ densely. Let $U\subset V$ open, $u\in U$ and suppose $\mathcal{G}\in \mathcal{C}^1(U;\R^m)$ possesses $H$-gradients $\nabla\mathcal{G}_1(u),\ldots, \nabla\mathcal{G}_m(u)$ in the sense of \Cref{def:Hgrad}. Then the following are equivalent.
	\begin{enumerate}[(i)]
		\item $\mathcal{G}^{\prime}(u)\colon V\to \R^m$ is surjective,
		\item $\nabla\mathcal{G}_1(u), \ldots,\nabla\mathcal{G}_m(u)$ are linearly independent in  $H$.
	\end{enumerate}
\end{lem}
\begin{proof}
	Assume (i) holds and let $\lambda\in \R^m$ be such that $\sum_{k=1}^m \lambda_k \nabla\mathcal{G}_k(u)=0$ in $H$. Then, for any $v\in V\subset H$ we have
	$0 = \sum_{k=1}^m \lambda_k \langle \nabla\mathcal{G}_k(u),v\rangle = \langle\lambda, \mathcal{G}^{\prime}(u)v\rangle_{\R^m}$ by \Cref{def:Hgrad}. Hence, $\lambda\in \left(\Image \mathcal{G}^{\prime}(u)\right)^{\perp_{\R^m}} = \{0\}$ by (i). Conversely, suppose (ii) holds and $\lambda\in \left(\Image \mathcal{G}^{\prime}(u)\right)^{\perp_{\R^m}}$. Then, we have
	\begin{align*}
	0 = \langle\lambda, \mathcal{G}^{\prime}(u)v\rangle_{\R^m} =\sum_{k=1}^m \lambda_k \langle\nabla\mathcal{G}_k(u), v\rangle = \left\langle\sum_{k=1}^m\lambda_k\nabla\mathcal{G}_k(u),v \right\rangle \text{ for all } v\in V.
	\end{align*}
	As a consequence, $\sum_{k=1}^m \lambda_k \nabla\mathcal{G}_k(u) = 0$ in $H$ by density of $V\subset H$, thus $\lambda=0$ by (ii).
\end{proof}
\begin{proof}[{Proof of \Cref{cor:main hilbert}}]
	Assumptions (i)-(iv) of \Cref{thm:main} are satisfied if we choose $Y=Y^{\ast}=H$ under the identification of $H$ with its image in $V^{\ast}$. Note that the extension of $\mathcal{G}^{\prime}(u)$ is given by
	\begin{align}\label{eq:extensionNablaG}
	\overline{\mathcal{G}^{\prime}(u)} y = \left(\langle \nabla\mathcal{G}_1(u), y\rangle, \dots,  \langle\nabla\mathcal{G}_m(u),y\rangle \right)^T \text{ for }y\in H.
	\end{align} 
	Thus, assumption (v) of \Cref{thm:main} is satisfied if and only if $(\nabla\mathcal{G}_k)^{\prime}(\bar{u})\colon V\to H$ is compact for all $k=1,\dots,m$ which is exactly assumption (v) in \Cref{cor:main hilbert}. By \Cref{lem:G'surjective}, assumption (vi) in \Cref{thm:main} is also satisfied. We conclude that there exists $C, \sigma>0$ and $\theta\in (0,\frac{1}{2}]$ such that for any $u\in \mathcal{M}$ with $\Norm{u-\bar{u}}{V}\leq \sigma$, we have
	\begin{align}\label{eq:BewCor1.7LS}
	\abs{\mathcal{E}(u)-\mathcal{E}(\bar{u})}^{1-\theta}\leq C \Norm{\mathcal{E}^{\prime}(u)}{\overline{\mathcal{T}_u\mathcal{M}}^{\ast}}.
	\end{align}
	By \Cref{prop:propertiesTuM}, $\overline{\mathcal{T}_u\mathcal{M}} = \ker\overline{\mathcal{G}^{\prime}(u)}$. 
	Consequently, by \eqref{eq:extensionNablaG}, we have 
	\begin{align*}
	\overline{\mathcal{T}_u\mathcal{M}} = \left\{ y\in H\mid \langle \nabla\mathcal{G}_k(u),y\rangle = 0 \text{ for all }k=1, \dots, m\right\}.
	\end{align*}
	Hence, if $P(u)\in \mathcal{L}(H)$ denotes the orthogonal projection onto $\overline{\mathcal{T}_u\mathcal{M}}\subset H$ we can estimate the right hand side of \eqref{eq:BewCor1.7LS} by
	\begin{align*}
	\Norm{\mathcal{E}^{\prime}(u)}{\overline{\mathcal{T}_u\mathcal{M}}^{\ast}} &= \sup_{0\neq y\in \overline{\mathcal{T}_u\mathcal{M}}}\frac{\mathcal{E}^{\prime}(u)y}{\Norm{y}{H}} = \sup_{0\neq y\in \overline{\mathcal{T}_u\mathcal{M}}} \frac{\langle\nabla\mathcal{E}(u),P(u)y\rangle}{\Norm{y}{H}}\\&
	= \sup_{0\neq y\in \overline{\mathcal{T}_u\mathcal{M}}} \frac{\langle P(u)\nabla\mathcal{E}(u),y\rangle}{\Norm{y}{H}} \leq \Norm{P(u)\nabla\mathcal{E}(u)}{H}.
	\end{align*}
	Together with \eqref{eq:BewCor1.7LS}, this proves \eqref{eq:LSHilbert}.
\end{proof}

In the setting of \Cref{cor:main hilbert}, we may deduce the following abstract convergence result for the associated gradient flow.
\begin{cor}\label{cor:ConvergenceTheorem}
	Let $\CalE, \mathcal{G}$ be as in \Cref{cor:main hilbert} with $m=1$ and suppose $u \in \CalC^1([0,\infty);V)$ is a solution of the \emph{constrained gradient flow equation}
	\begin{align}\label{eq:constrGradFlow}
	\left\{\begin{array}{lll}
	\partial_t u &= - \nabla\mathcal{E}(u)+\lambda(u)\nabla\mathcal{G}(u), & t>0\\
	u(0)&=u_0,&
	\end{array}\right.,
	\end{align}
	where $\lambda(u)$ is as in \eqref{eq:LojaHilbertLambda}. Assume that $\overline{\{u(t)\mid t\geq 0\}}\subset V$ is compact. Then $\lim_{t\to\infty}u(t)$ exists in $V$.
\end{cor}
\begin{rem}
	The key idea in the proof of \Cref{cor:ConvergenceTheorem}  is the following (formal) computation, based on \cite{Simon83}. If $\CalE\vert_{\mathcal{G}^{-1}\{0\}}$ satisfies a refined {\L}ojasiewicz--Simon gradient inequality near $\bar{u}\in \overline{\{u(t)\mid t\geq 0\}}$, then
	\begin{align*}
	-\frac{\diff}{\diff t} \left(\CalE(u)- \CalE(\bar{u})\right)^{\theta} &= -\theta \left(\CalE(u)-\CalE(\bar{u})\right)^{\theta-1} \langle \nabla\CalE(u), \partial_t u\rangle \\
	&=\theta \left(\CalE(u)-\CalE(\bar{u})\right)^{\theta-1} \norm{\nabla\CalE(u)-\lambda(u)\nabla\mathcal{G}(u)}{H}\norm{\partial_t u}{H} \\
	&\geq \frac{\theta}{C}\norm{\partial_t u}{H}.
	\end{align*}
	This implies $\partial_t u \in L^1([0, \infty);H)$ which yields the claim (see \cite[Theorem 12.2]{ChillISEM} for a detailed presentation of this argument, with weaker regularity assumptions).
\end{rem}

\section{Optimality discussion}\label{sec:Optimality}
In this section, we will discuss why the assumptions in \Cref{thm:main} and \Cref{cor:main hilbert} cannot be omitted.

First, we provide an example, inspired by the Hilbert space case in \cite[Theorem 2.1]{Haraux11}, which implies that in any Banach space of infinite dimension, there will exist an energy which fails to satisfy the {\L}ojasiewicz--Simon gradient inequality. The construction relies on the following nontrivial fact.
\begin{thm}\label{thm:ExistBasicSeq}
	Let $V$ be a Banach space of infinite dimension and let $\varepsilon>0$. Then there exist sequences $(e_n)_{n\in\N}\subset V$ with $\norm{e_n}{}=1$ for all $n\in \N$ and $(\phi_k)_{k\in \N}\subset V^{\ast}$ with $\norm{\phi_k}{}\leq  2(1+\varepsilon)$ for all $k\in \N$ such that 
	\begin{align*}
	\phi_k(e_n)=\delta_{k,n}, \quad\text{ for all }k,n\in \N.
	\end{align*}
\end{thm}
\begin{proof}
	See \Cref{sec:BasicSeq}.
\end{proof}
\begin{example}\label{ex:LojaFailGenericBS}
	Let $V$ be a Banach space of infinite dimension and $\varepsilon>0$. Let $(e_n)_{n\in N}$ and $(\phi_k)_{k\in \N}$ be as in \Cref{thm:ExistBasicSeq}. Let $\lambda\in \ell^{1}(\N)$ with $\lambda_k\neq 0$ for all $k\in \N$. Then, $x=0$ is a critical point of the analytic energy
	\begin{align*}
	\mathcal{E}(x) \defeq \frac{1}{2} \sum_{k=1}^{\infty} \lambda_k \abs{\phi_k(x)}^2, \text{ for } x\in V,
	\end{align*}
	but $\mathcal{E}$ satisfies no {L}ojasiewicz--Simon gradient inequality around $x=0$.
\end{example}
\begin{proof}
	First, we will prove that $\mathcal{E}$ is analytic. Indeed, we have $\mathcal{E}(x) = \Phi(x,x)$, where $\Phi(x,y) \defeq \sum_{k=1}^{\infty}\lambda_k \phi_k(x)\phi_k(y)$ for $x,y \in V$. Note that $\Phi \colon V\times V\to \R$ is bilinear and bounded by
	\begin{align*}
	\abs{\Phi(x,y)}\leq \beta^2 \sum_{k=1}^{\infty}\abs{\lambda_k} \norm{x}{}\norm{y}{}\text{ for all }x,y \in V,
	\end{align*}
	where $\beta\defeq 2(1+\varepsilon)$.
	By \Cref{ex:multilinAnalytic}, $\Phi$ and hence $\mathcal{E}$ is analytic. Furthermore, we find
	\begin{align*}
	\mathcal{E}^{\prime}(x)(y) = \sum_{k=1}^{\infty} \lambda_k \phi_k(x)\phi_k(y) \text{ for all }x,y \in V.
	\end{align*}
	Clearly, $\mathcal{E}^{\prime}(0)=0$. For $k\in \N$ and $t>0$ we have
	\begin{align*}
	\Norm{\mathcal{E}^{\prime}(te_n)}{V^{\ast}} = \sup_{\norm{y}{V}=1} \sum_{k=1}^{\infty} \lambda_k \phi_k(te_n) \phi_k(y) \leq t\abs{\lambda_n} \Norm{\phi_n}{V^{\ast}} \leq \beta t \abs{\lambda_n}
	\end{align*}
	and $\mathcal{E}(t e_n) = \frac{1}{2}\lambda_n t^2$. Thus, if $\mathcal{E}$ satisfied a {\L}ojasiewicz--Simon gradient inequality for some $C, \sigma, \theta>0$, for all $n\in\N$ and $0<t\leq \sigma$ we would get
	\begin{align*}
	\left(\frac{1}{2}\abs{\lambda_n} t^2\right)^{1-\theta} \leq C\beta \abs{\lambda_n} t.
	\end{align*}
	Dividing by $\abs{\lambda_n}$ and letting $n \to\infty$ yields a contradiction, since $\lambda_n\to 0$ as $\lambda\in \ell^1(\N)$.
\end{proof}
%\changed{This example illustrates that more than analyticity is necessary, even in the absence of a constraint, and hence justifies conditions (i)-(iii) in \Cref{thm:main}.}
It is not too difficult to see that the second derivative of $\CalE$ in \Cref{ex:LojaFailGenericBS} fails to be Fredholm. Consequently, condition (iii) in \Cref{thm:FredholmLoja} (and \Cref{thm:main}) is violated, whereas conditions (i)-(ii) are satisfied with $Z=H$ ($Y=H$, respectively), {indicating} that mere analyticity of the energy is not enough.
In fact, the following result shows that it is never sufficient in infinite dimensions.
\begin{cor}
	Let $V$ be a Banach space, and let $U\subset V$ be open. Then, $\dim V<\infty$ if and only if every analytic function $\mathcal{E}\in \mathcal{C}^{\omega}(U;\R)$ satisfies a {\L}ojasiewicz--Simon gradient inequality at each of its critical points.
\end{cor}
For the sake of simplicity, throughout the rest of this section we restrict ourselves to the Hilbert space case in \Cref{cor:main hilbert} with $V=H$. That way, assumptions {(i),} (ii) and (iv) are automatically satisfied if the energy and the constraint are analytic. The next example shows that we can not drop the compactness assumption (v) in \Cref{cor:main hilbert}.

\begin{example}\label{ex:NeedCompactnessG''}
	Consider the Hilbert space $H=\R\times\ell^2(\N)$ and let $\lambda\in \ell^1(\N)$. We write elements $x\in H$ as $x=(x_0, x^{\prime})$ with $x^{\prime}\in \ell^2(\N)$. The natural norm on $H$ is given by $\norm{x}{H}^2 \defeq \abs{x_0}^{2} + \norm{x^{\prime}}{\ell^2(\N)}^2$.
	For $x = (x_0, x^{\prime})\in H$ define 
	\begin{align*}
	\mathcal{E}(x) \defeq x_0 + \sum_{n=1}^{\infty} \abs{x_n^{\prime}}^2 \text{ for } x\in H.
	\end{align*}
	Then $\mathcal{E}$ satisfies assumptions (i)-(iii) in \Cref{cor:main hilbert} with $V=H$. We define 
	\begin{align*}
	\psi\colon \ell^2(\N)\to\R, \psi(x^{\prime})\defeq \sum_{n=1}^{\infty} (\lambda_n-1)\abs{x_n^{\prime}}^2
	\end{align*} and consider $\mathcal{G}\colon H\to \R,$ $\mathcal{G}(x) \defeq x_0-\psi(x^{\prime})$ and $\mathcal{M}\defeq \mathcal{G}^{-1}(\{0\})$. Then, $\mathcal{E}\vert_{\mathcal{M}}$ does not satisfy a refined {\L}ojasiewicz--Simon gradient inequality at the origin, but satisfies all assumptions of \Cref{cor:main hilbert} with $V=H$ except assumption (v).
\end{example}
\begin{proof} It is easy to see that $\mathcal{E}$, $\psi$ and $\mathcal{G}$ are analytic. {Given $\bar{x},y\in H$, a short computation yields $(\nabla \CalE)'(\bar x) y = (0, 2y')$, so the second derivative is Fredholm with index zero and the first part of the statement is proven.}
	
	Moreover, $\mathcal{G}$ possesses an $H$-gradient $ \nabla\mathcal{G}\left(x_0, x^{\prime}) = (1, \left(2(\lambda_n -1)x_n^{\prime}\right)_{n\in \N}\right)\in H$ and the gradient map is analytic. Also note that $\mathcal{G}(0)=0$ and $\nabla\mathcal{G}(0)=(1,0)$, so assumptions (iv) and (vi) of \Cref{cor:main hilbert} are satisfied. By \Cref{ex:graphManifold}, $\mathcal{M}$ is an analytic submanifold of $H$ near the origin, with a single chart $\varphi \colon \ell^2(\N)\to H$, $\varphi(x^{\prime}) \defeq (\psi(x^{\prime}), x^{\prime})\in H$ which coincides with the chart from \Cref{thm:localgraph} in this example.
	However, note that the operator $T\defeq (\nabla\mathcal{G})^{\prime}(0)\colon H\to H$, $(\nabla\mathcal{G})^{\prime}(0)(x_0, x^{\prime})=\left( 0, \left(2(\lambda_n-1)x^{\prime}_n\right)_{n\in \N}\right)$ is not compact. Indeed, let $x_k \defeq (0, e_k^{\prime})\in H$ where $e_k^{\prime}\in \ell^2(\N)$ is the standard $k$-th unit vector. A short computation yields $T x_k \rightharpoonup 0$, however $T x_k \not \to 0$, since $\norm{Tx_k}{H} = 2 \abs{\lambda_k-1} \to 1$, so $T$ cannot be compact. For $x^{\prime}\in \ell^2(\N)$ we have
	\begin{align*}
	(\mathcal{E}\circ \varphi)(x^{\prime}) &= \psi(x^{\prime}) + \sum_{n=1}^{\infty}\abs{x_n^{\prime}}^2 = \sum_{n=1}^{\infty}\lambda_n \abs{x_n^{\prime}}^2.
	\end{align*}
	Similar to \Cref{ex:LojaFailGenericBS}, one can show that $\mathcal{E}\circ\varphi$ does not satisfy a {\L}ojasiewicz--Simon gradient inequality at the origin $x^{\prime}=0$ by assuming the inequality holds and then testing it with $x^{\prime} = e^{\prime}_k \in \ell^2(\N)$ for all $k\in \N$. \Cref{lem:FLojaIffE_MLoja} then implies that $\mathcal{E}\vert_{\mathcal{M}}$ cannot satisfy a refined {\L}ojasiewicz--Simon gradient inequality at $x=0$ either.\footnote{At this point it is crucial that in \Cref{lem:FLojaIffE_MLoja} we did not require assumption (v) of \Cref{thm:main} to be satisfied.}
\end{proof}

\Cref{ex:NeedCompactnessG''} shows that assumption (v) in \Cref{cor:main hilbert} cannot be omitted. Note that while one can easily show {using  \Cref{thm:FredholmLoja}} that $\mathcal{E}$ as in \Cref{ex:NeedCompactnessG''} satisfies a {\L}ojasiewicz--Simon gradient inequality, $\mathcal{E}\vert_{\mathcal{M}}$ does not satisfy the refined inequality \eqref{eq:LSnonlin}. In particular, the property of satisfying a {\L}ojasiewicz--Simon gradient inequality does in general not behave well under the restriction to a submanifold, even if we assume finite codimension.

Let us  {also} %finally
remark that condition (vi) in \Cref{cor:main hilbert} is in general necessary to guarantee that $\mathcal{M}$ is a manifold. 
{Finally, note that our main result only considers submanifolds of finite codimension. The following example shows that our main result cannot be extended to the case of infinite codimension, even for linear subspaces.}
\begin{example}\label{ex:InfiniteCodimension}
	Consider the Hilbert space $H\defeq \ell^2(\N)\times \ell^2(\N)$ with
	\begin{align*}
	\langle(x,x'), (y,y')\rangle_H \defeq \langle x,y\rangle_{\ell^2(\N)} + \langle x',y'\rangle_{\ell^2(\N)}, \text{ for }(x,x'), (y,y')\in H.
	\end{align*}
	Then, the energy $\CalE\colon H \to\R, \CalE(x, x')\defeq \sum_{n=1}^\infty  \abs{x_n}^2-\abs{x_n'}^2$ satisfies assumptions (i)-(iii) in \Cref{cor:main hilbert} with $V=H$. For the constraint function $\mathcal{G}\colon H \to\ell^2(\N)$, $(x,x')\mapsto (x_n -\sqrt{1+n^{-2}}x_n')_{n\in\N}$, the set $\mathcal{M}\defeq\mathcal{G}^{-1}(\{0\})$ is a linear subspace of $H$ and $(x,x')=(0,0)$ is a constrained critical point of $\CalE\vert_{\mathcal{M}}$. Moreover, $\mathcal{G}$ and $\CalE$ are analytic, $\nabla\mathcal{G}(0,0)\neq 0$ and $0=\left(\nabla\mathcal{G}\right)'(0,0)\colon H \to H$ is compact. However, $\CalE\vert_{\mathcal{M}}$ does not satisfy a refined {\L}ojasiewicz--Simon gradient inequality near the origin.
\end{example} 
\begin{proof}
	The properties of $\CalE$ and $\mathcal{G}$ can be shown as in in \Cref{ex:NeedCompactnessG''}. {The natural chart for $\mathcal{M}$ is given by 
		\begin{align*}
		\varphi\colon \ell^2(\N)\to H, \varphi(y)\defeq \left( (\sqrt{1+n^{-2}} y_n)_{n\in \N}, (y_n)_{n\in\N}\right),
		\end{align*}
		and we observe that $(\CalE\circ \varphi)(y) = \sum_{n=1}^{\infty} \frac{1}{n^2}\abs{y_n}^2$ cannot satisfy a {\L}ojasiewicz--Simon gradient inequality near $x=0$ by \Cref{ex:LojaFailGenericBS}}. Hence, by an argument similar to \Cref{lem:FLojaIffE_MLoja}, neither does $\CalE\vert_{\mathcal{M}}$.
	%		Precomposing $\CalE$ with the natural chart yields an energy similar to \Cref{ex:LojaFailGenericBS}, which does not satisfy a {\L}ojasiewicz--Simon gradient inequality, and hence, by an argument similar to \Cref{lem:FLojaIffE_MLoja}, neither does $\CalE\vert_{\mathcal{M}}$.
\end{proof}

\section{Applications}\label{sec:application}
{In this section, we will apply our result from \Cref{thm:main} to different energies on Sobolev spaces with \emph{isoperimetric constraints} (cf. \cite[Chapter 2.1]{GiaquintaHildebrandt}). Like in \Cref{cor:ConvergenceTheorem} this can be then used to conclude convergence for precompact solutions of the associated gradient flows.}
\subsection{Surface area with an isoperimetric constraint}\label{subsec:isoperimetric}
%
%In this section, we will apply our result from \Cref{thm:main} to the surface area energy of a graph with prescribed boundary and an \emph{isoperimetric constraint} (cf. \cite[Chapter 2.1]{GiaquintaHildebrandt}).

Throughout this subsection, we assume that $\Omega\subset \R^d$ is a domain with $\mathcal{C}^{1,1}$-boundary. We want to study the \emph{surface area} or $d$-dimensional Hausdorff measure of $\graph(u)\subset \R^{d+1}$ given by
\begin{align*}
\mathcal{E}(u) \defeq  \int_{\Omega}\sqrt{1+|\nabla u|^2}\diff x.
\end{align*}
Note that while this energy is already defined if we merely require $u\in W^{1,1}(\Omega)$, a natural space to study a $L^2$-gradient flow would be $W^{2,2}(\Omega)$. However, we consider $u\in V\defeq W^{2,p}(\Omega)\cap W^{1,p}_0(\Omega)$ with $d<p<\infty$ and $Y \defeq L^q(\Omega)$ where $\frac{1}{p}+\frac{1}{q}=1$. The condition on $p$ and our choice of spaces will imply analyticity (cf. \Cref{rem:HilbertBad}). We want to study  $\mathcal{E}$ on the set of functions which satisfy the constraint
\begin{align}\label{eq:DefConstraintG}
\mathcal{G}(u) = \int_{\Omega} g(u)\diff x = \int_{\Omega} g(u(x))\diff x = 0,
\end{align}
where $g\colon\R\to\R$ is an analytic function. Note that the energy as well as the constraint are well defined since $W^{2,p}(\Omega)$ embeds into both $W^{1,1}(\Omega)$ and $\mathcal{C}(\overline{\Omega})$ by \cite[Corollary 7.11]{GilbargTrudinger}. Moreover, $V\hookrightarrow L^q(\Omega)$ densely, so we get an induced embedding $L^p(\Omega)= Y^{\ast}\hookrightarrow V^{\ast}$. 
%Recall that the duality of $L^p$ and $L^q$ is given by
%\begin{align}\label{eq:dualityPQ}
%\langle u, v\rangle_{L^p L^q} = \int_{\Omega} uv\diff x, \text{ for }u\in L^p(\Omega), v\in L^q(\Omega).
%\end{align}

We recall the following important property of Nemytskii operators.
\begin{thm}[{\cite[Theorem 6.8]{Appell90}}]\label{thm:nemytskiiAnalytic}
	Let $F\in\mathcal{C}(\R)$. Then, the \emph{superposition operator} $\mathcal{F}\colon \mathcal{C}(\overline{\Omega})\to\mathcal{C}(\overline{\Omega})$, $\mathcal{F}(v) ={F}(v)$ is analytic if and only if the function $F$ is.
\end{thm}
\begin{lem}\label{lem:Ganalytic}
	The map $\mathcal{G}\colon V\to\R$ is analytic with
	\begin{align}\label{eq:G'2}
	\mathcal{G}^{\prime}(u) v = \int_{\Omega} g^{\prime}(u) v\diff x, \text{ for }u, v\in V.
	\end{align}
	In particular, $\mathcal{G}^{\prime}(u) = g^{\prime}(u)\in Y^{\ast}=L^p(\Omega)$.
\end{lem}
\begin{proof}
	By the embedding $W^{2,p}(\Omega)\hookrightarrow \mathcal{C}(\overline{\Omega})$, the map $V\ni u\mapsto u\in \mathcal{C}(\overline{\Omega})$ is analytic. Hence, so is $V\ni u\mapsto g(u)\in \mathcal{C}(\overline{\Omega})$ by \Cref{thm:nemytskiiAnalytic}. Integrating is analytic by \Cref{ex:multilinAnalytic}, since it is linear and bounded. Using \Cref{thm:companalytic}, this yields $\mathcal{G}\in \mathcal{C}^{\omega}(V;\R)$. Furthermore, \eqref{eq:G'2} follows, since for $u,v\in V$, we have
	\begin{align*}
	\mathcal{G}^{\prime}(u)v = \dtzero \mathcal{G}(u+tv) = \int_{\Omega} \dtzero g(u+tv)\diff x = \int_{\Omega} g^{\prime}(u) v\diff x.
	\end{align*}
	Clearly, $\mathcal{G}^{\prime}(u) = g^{\prime}(u)\in Y^{\ast}=L^p(\Omega)$, since $g^{\prime}(u) \in\mathcal{C}(\overline{\Omega})\subset L^p(\Omega)$. Moreover, we have $\mathcal{G}^{\prime}\in \mathcal{C}^{\omega}(V, Y^{\ast})$, since $g^{\prime}\colon\R\to\R$ is analytic and so is $V\ni u\mapsto g^{\prime}(u) \in L^p(\Omega)$ by the embeddings $V\hookrightarrow  \mathcal{C}(\overline{\Omega})$ and $ \mathcal{C}(\overline{\Omega})\hookrightarrow L^p(\Omega)$ and using \Cref{thm:companalytic,thm:nemytskiiAnalytic}.
\end{proof}
As a next step, we compute the second derivative of $\mathcal{G}$.
\begin{lem}\label{lem:G''compact}
	The operator $\mathcal{G}^{\prime\prime}(u)\colon V\to Y^{\ast}$ is compact for any $u\in V$.
\end{lem}
\begin{proof}
	Let $u,v \in V$. From \eqref{eq:G'2}, we conclude 
	\begin{align*}
	\mathcal{G}^{\prime\prime}(u)v = \dtzero g^{\prime}(u+tv) = g^{\prime\prime}(u) v \in L^p(\Omega).
	\end{align*}
	Since $u\in V\hookrightarrow \mathcal{C}(\overline{\Omega})$, $g^{\prime\prime}(u)\in \mathcal{C}(\overline{\Omega})$ and hence $V\ni v\mapsto g^{\prime
		\prime}(u)v\in Y^{\ast}=L^p(\Omega)$ is compact, since the embedding $V\ni v\mapsto v\in L^p(\Omega)$ is compact. It follows that $\mathcal{G}^{\prime\prime}(u)\colon V\to Y^{\ast}$ is compact.
\end{proof}

\begin{lem}\label{lem:EGAnalytic}
	The map $\mathcal{E}\colon V\to\R$ is analytic with
	\begin{align}
	\mathcal{E}^{\prime}(u) v &= -\int_{\Omega} \divergence\left(\frac{\nabla u}{\sqrt{1+\abs{\nabla u}^2}}\right) v\diff x, \text{ for } u,v \in V. \label{eq:E'}
	\end{align}
	Moreover, $\mathcal{E}^{\prime}(u)\in Y^{\ast} = L^p(\Omega)$ for all $u\in V$.
\end{lem}
This analyticity statement motivates our choice of spaces in the beginning of this section.
\begin{proof}
	We first note that the following maps are analytic.
	\begin{enumerate}[(i)]
		\item The embedding $i\colon V\hookrightarrow \mathcal{C}^1(\overline{\Omega})$ since $d>p$ using \Cref{ex:multilinAnalytic}.
		\item The map $\mathcal{C}^1(\overline{\Omega})\to\mathcal{C}(\overline{\Omega}), u\mapsto \abs{\nabla u}^2$ by \Cref{ex:multilinAnalytic}, since it is the diagonal of a bounded bilinear map.
		\item The map 
		$\mathcal{C}(\overline{\Omega};(-1,\infty))\subset \mathcal{C}(\overline{\Omega})  \to \mathcal{C}(\overline{\Omega}), \mathcal{F}(v)=(1+v)^{-\alpha}$
		for $\alpha>0$ by \Cref{thm:nemytskiiAnalytic}, since the map $F\colon (-1,\infty)\to\R, F(x) = (1+x)^{-\alpha}$ is analytic.
		\item The map 
		\begin{align}\label{eq:alphaAnalytic}
		V \to \mathcal{C}(\overline{\Omega}), u\mapsto (1+\abs{\nabla u}^2)^{-\alpha}
		\end{align}
		for $\alpha>0$ as a composition of the maps in (i)-(iii) using \Cref{thm:companalytic}.
		\item The map $\mathcal{C}(\overline{\Omega})\to\R, v\mapsto\int_{\Omega}v\diff x$ by \Cref{ex:multilinAnalytic}, since it is linear and bounded.
	\end{enumerate}
	Since $\mathcal{E}$ can be written as the composition of these maps, $\mathcal{E}$ is analytic. For a proof of \eqref{eq:E'}, consider \cite[Chapter 1, 2.2 Example 5]{GiaquintaHildebrandt}. Note that $\mathcal{E}^{\prime}(u)\in L^p(\Omega)$, since for $u\in W^{2,p}(\Omega)$, we have using summation convention
	\begin{align}\label{eq:E'expanded}
	\begin{split}
	\divergence\left(\frac{\nabla u}{\sqrt{1+\abs{\nabla u}^2}}\right) &=  \frac{\Delta u}{\sqrt{1+\abs{\nabla u}^2}} -\partial_i u \frac{\partial_i\partial_j u \partial_j u}{(1+\abs{\nabla u}^2)^{\frac{3}{2}}}.
	\end{split}
	\end{align}
	By the embedding $W^{2,p}(\Omega)\hookrightarrow \mathcal{C}^1(\overline{\Omega})$ and since the denominators are bounded from below, we conclude that $\mathcal{E}^{\prime}(u)\in L^p(\Omega)$ for $u\in W^{2,p}(\Omega)$.
\end{proof}
\begin{lem}\label{lem:E'analytic}
	The function $\mathcal{E}^{\prime}\colon V\to L^p(\Omega), u\mapsto\mathcal{E}^{\prime}(u)$ is analytic.
\end{lem}
\begin{proof}
	The following maps are analytic.
	\begin{enumerate}[(i)]
		\item The map $V\to L^p(\Omega), u\mapsto \partial_i \partial_j u$ for any $i,j\in\{ 1, \dots, d\}$ by \Cref{ex:multilinAnalytic}.
		\item The maps $u \mapsto (1+\abs{\nabla u}^2)^{-\frac{1}{2}}, u\mapsto(1+\abs{\nabla u}^2)^{-\frac{3}{2}}\colon V\to \mathcal{C}(\overline{\Omega})$ by \eqref{eq:alphaAnalytic}.
		\item $V\to \mathcal{C}(\overline{\Omega}), u\mapsto \partial_j u$ for any $j\in \{1,\dots, d\}$ by \Cref{ex:multilinAnalytic}.
	\end{enumerate}
	Since the pointwise multiplications $\mathcal{C}(\overline{\Omega})\times L^p(\Omega)\to L^p(\Omega)$ and $\mathcal{C}(\overline{\Omega})\times \mathcal{C}(\overline{\Omega})\to \mathcal{C}(\overline{\Omega})$  are bilinear and bounded, they are analytic by \Cref{ex:multilinAnalytic}. Hence, so is $u\mapsto \mathcal{E}^{\prime}(u)$ by \eqref{eq:E'expanded} and \Cref{thm:companalytic}.
\end{proof}
\begin{lem}\label{lem:ExampleE''Fredholm}
	Let $u\in V$. Then the Fréchet derivative $\mathcal{E}^{\prime\prime}(u)\colon V\to Y^{\ast}$ is Fredholm of index zero.
\end{lem}
\begin{proof}
	By \Cref{lem:E'analytic}, for $u, v\in V,$  we compute
	\begin{align*}
	\mathcal{E}^{\prime\prime}(u)v &= \dtzero \mathcal{E}^{\prime}(u+tv) = -\divergence\left( \frac{\nabla v \sqrt{1+\abs{\nabla u}^2} - \nabla u \frac{\langle\nabla u, \nabla v\rangle}{\sqrt{1+\abs{\nabla u}^2}}}{1+\abs{\nabla u}^2} \right) \\
	&= - \frac{\Delta v}{\sqrt{1+\abs{\nabla u}^2}}+ \frac{1}{(1+\abs{\nabla u}^2)^{\frac{3}{2}}} \left\langle \nabla u, \nabla\langle \nabla u, \nabla v\rangle \right\rangle + \tilde{K}v\\
	%&=-\frac{\Delta v}{\sqrt{1+\abs{\nabla u}^2}} + \frac{1}{(1+\abs{\nabla u}^2)^{\frac{3}{2}}}\partial_i u \partial_j u \partial_i\partial_j v + Kv \eqdef -Av + Kv,\\
	&= - \frac{1}{(1+\abs{\nabla u}^{2})^{\frac{3}{2}}}\left(\delta_{ij}(1+\abs{\nabla u}^{2}) - \partial_i u \partial_j u\right) \partial_{i}\partial_j v + Kv \eqdef - Av + Kv
	\end{align*}
	using summation convention, where $\tilde{K}, K\colon V\to L^p(\Omega)$ only contain terms in $v$ of order $1$ or lower, whence are compact by the {Rellich--Kondrachov} Theorem \cite[Theorem 7.26]{GilbargTrudinger}.
	
	It is easy to see that $A$ uniformly is elliptic, hence $A \colon W^{2,p}(\Omega)\cap W^{1,p}_0(\Omega)\to L^p(\Omega)$ is an isomorphism by \cite[Theorem 9.15]{GilbargTrudinger}. Therefore, $\mathcal{E}^{\prime\prime}(u) = -A+K\colon V\to Y^{\ast}$ is Fredholm of index zero by \Cref{prop:Fredholmproperties}. 
\end{proof}

Now, we can apply \Cref{thm:main} to our situation.

\begin{thm}\label{thm:ExampleLoja}
	Let $\bar{u}\in V, \mathcal{G}(\bar{u})=0$ with $g^{\prime}(\bar{u})\not \equiv 0$ be a constraint critical point of $\mathcal{E}$ on  $\mathcal{M} = \{u\in V\mid \mathcal{G}(u)=0\}$. Then, $\mathcal{M}$ is locally a manifold near $\bar{u}$ and satisfies a {\L}ojasiewicz--Simon gradient inequality on $\mathcal{M}$, i.e. there exist $C, \sigma>0$ and $\theta\in (0,\frac{1}{2}]$ such that for any $u\in \mathcal{M}$ with $\Norm{u-\bar{u}}{W^{2,p}(\Omega)}\leq \sigma$, we have
	\begin{align}\label{eq:ExampleLojaEstimate}
	\abs{\mathcal{E}(u)-\mathcal{E}(\bar{u})}^{1-\theta}\leq C\Big\Vert{\mathcal{E}^{\prime}(u)-\frac{\int_{\Omega} \mathcal{E}^{\prime}(u) g^{\prime}(u)\diff x}{\int_\Omega (g^{\prime}(u))^2 \diff x} g^{\prime}(u)}\Big\Vert_{L^p(\Omega)}.
	\end{align}
\end{thm}
\begin{proof}
	We verify that \Cref{thm:main} is applicable with $U=V=W^{2,p}(\Omega)\cap W^{1,p}_0(\Omega)$ and $Y= L^q(\Omega)$, so $Y^{\ast}=L^p(\Omega)$. Analyticity of $\mathcal{G}$ and $\mathcal{E}$ has been proven in \Cref{lem:Ganalytic,lem:EGAnalytic}. Clearly $V=W^{2,p}(\Omega)\cap W^{1,p}_0(\Omega)\hookrightarrow  Y^{\ast}$ densely. Moreover, $\mathcal{E}^{\prime}\in \mathcal{C}^{\omega}(V;Y^{\ast})$ by \Cref{lem:E'analytic}. The Fredholm property of $\mathcal{E}^{\prime\prime}(\bar{u})\colon V\to Y^{\ast}$ has been established in \Cref{lem:ExampleE''Fredholm}. By \Cref{lem:Ganalytic}, $\mathcal{G}^{\prime}$ extends analytically in the sense of assumption (iv) in \Cref{thm:main}. Moreover, the Fréchet derivative, $\mathcal{G}^{\prime\prime}(\bar{u})\colon V\to Y^{\ast}$ is compact by \Cref{lem:G''compact}. By assumption, $\mathcal{G}^{\prime}(\bar{u}) = g^{\prime}(\bar{u})\not\equiv 0$, hence it is surjective as an operator $Y\to\R$.
	Thus, by \Cref{thm:main}, $\mathcal{M}$ is locally a manifold near $\bar{u}$ and
	there exist $C, \sigma>0$ and $\theta\in (0, \frac{1}{2}]$ such that for $u\in \mathcal{M}$ with $\Norm{u-\bar{u}}{W^{2,p}}\leq \sigma$, we have
	\begin{align}\label{eq:ExampleLojaEstimate1}
	\abs{\mathcal{E}(u)-\mathcal{E}(\bar{u})}^{1-\theta}\leq C\Norm{\mathcal{E}^{\prime}(u)}{\overline{\mathcal{T}_u\mathcal{M}}^{\ast}}.
	\end{align}
	It remains to conclude \eqref{eq:ExampleLojaEstimate} from \eqref{eq:ExampleLojaEstimate1}. To that end, note that by \Cref{prop:propertiesTuM} we have $\overline{\mathcal{T}_u\mathcal{M}} = \ker\overline{\mathcal{G}^{\prime}(u)} = \{ w\in L^q(\Omega) \mid \int_\Omega g^{\prime}(u) w\diff x =0 \}$. Hence, for any $\lambda \in \R$, we have
	\begin{align}\label{eq:lambdaminimal}
	\Norm{\mathcal{E}^{\prime}(u)}{\overline{\mathcal{T}_u\mathcal{M}}^{\ast}} &= \sup_{\substack{w\in \ker\overline{\mathcal{G}^{\prime}(u)}}} \frac{\mathcal{E}^{\prime}(u)w}{\Norm{w}{L^q}} = \sup_{\substack{w\in \ker\overline{\mathcal{G}^{\prime}(u)}}} \frac{\int_{\Omega} (\mathcal{E}^{\prime}(u) - \lambda g^{\prime}(u) )w\diff x}{\Norm{w}{L^q}}\nonumber \\
	&\leq \sup_{w\in L^q(\Omega)} \frac{\int_\Omega(\mathcal{E}^{\prime}(u) -\lambda g^{\prime}(u))w\diff x}{\Norm{w}{L^q(\Omega)}} = \Norm{\mathcal{E}^{\prime}(u) - \lambda g^{\prime}(u)}{L^p(\Omega)}.
	\end{align}
	Choosing $\lambda \defeq \frac{\int_\Omega \mathcal{E}^{\prime}(u)g^{\prime}(u)\diff x}{\int_{\Omega}(g^{\prime}(u))^2\diff x}$ yields \eqref{eq:ExampleLojaEstimate}. 
\end{proof}
\begin{rem}
	\begin{enumerate}
		\item If $g(x) = x- \Gamma$ for some $\Gamma>0$ in \eqref{eq:DefConstraintG}, the energy $\mathcal{E}\vert_{\mathcal{M}}$ corresponds to the restriction of the surface area of $\graph(u)$ on the set of graphs  with fixed enclosed volume $\Gamma$ with the $\R^d\times\{0\}$-hyperplane.
		\item By considering the shifted energies $\tilde{\mathcal{E}}(u) = \mathcal{E}(u+\beta)$
		and $\tilde{\mathcal{G}}(u) = \mathcal{G}(u+\beta)$ for $u\in V=W^{2,p}(\Omega)\cap W^{1,p}_0(\Omega)$ and fixed $\beta\in W^{2,p}(\Omega)$, the result can be extended to general Dirichlet boundary data.
		\item Notice that in the proof of {\Cref{thm:ExampleLoja}}, we have some freedom in the choice of $\lambda$. Our choice is justified, since then any solution $u=u(t)$ of the equation
		\begin{align*}
		\partial_t u = -\mathcal{E}^{\prime}(u) + \lambda(u)\mathcal{G}^{\prime}(u)
		\end{align*}
		will preserve the constraint, i.e. $\mathcal{G}(u(t))=0$ for all $t$, provided $u$ is smooth enough.
		\item It is not clear, whether our choice of $\lambda$ is optimal in the sense that it minimizes the right hand side of \eqref{eq:lambdaminimal}. However, in the case $p=2$, so $d=1$, our choice of $\lambda$ yields the orthogonal projection (cf. \Cref{rem:HilbertSpace m=1}), and hence minimizes the right hand side of \eqref{eq:lambdaminimal}.
	\end{enumerate}
\end{rem}
\subsection{The Allen--Cahn equation}\label{subsec:AllenCahn}
The following reaction-diffusion equation plays an important role in mathematical physics, modeling the process of phase separation \cite{CahnHilliard},
\begin{align}\label{eq:AllenCahn}
\partial_t u  &= \varepsilon\Delta u + f(u) \text{ on }\Omega\times (0,T).
\end{align}
Here, $T, \varepsilon>0$ and $\Omega \subset \R^d$ is a domain with $\mathcal{C}^{1,1}$-boundary.
\Cref{eq:AllenCahn} is the $L^2$-gradient flow of the \emph{Ginzburg--Landau free energy}
\begin{align}\label{eq:energyAllenCahn}
\mathcal{E}_{\varepsilon} (u) \defeq \int_{\Omega} \left(\frac{\varepsilon}{2}\abs{\nabla u}^2 + W(u)\right) \diff x,
\end{align}
where $f(u)= -W^{\prime}(u)$. The function $W$ describes some potential, a common choice is $W(s) = \frac{1}{4\varepsilon}(1-s^2)^2,$ the \emph{double well potential}.

In their celebrated works \cite{Modica87,ModicaMortola77},  L. Modica and S. Mortola proved that as $\varepsilon\to 0$, the energy $\mathcal{E}_{\varepsilon}$ in \eqref{eq:energyAllenCahn} $\Gamma$-converges to the perimeter of a suitable level set.
For our result, we will consider $\varepsilon>0$ as being fixed, and therefore, we can assume $\varepsilon=1$ without loss of generality. {We} define $V\defeq W^{2,2}(\Omega)\cap W^{1,2}_0(\Omega)$ and write $\mathcal{E}\defeq \mathcal{E}_{1}$.

In this subsection, we will use \Cref{thm:main} to establish a {\L}ojasiewicz--Simon gradient inequality for the constrained energy $\mathcal{E}\vert_{\mathcal{M}}$, where
$\mathcal{M} = \{u\in V\mid \mathcal{G}(u)=0\}$ and
\begin{align*}
\mathcal{G}(u) \defeq \int_{\Omega} g(u)\diff x, 
\end{align*}
for a continuous function $g\colon\R\to\R$, cf. \Cref{subsec:isoperimetric}. 
\begin{thm}\label{thm:AllenCahnLoja}
	Let $d\leq 3$, $W, g\in \mathcal{C}^{\omega}(\R;\R)$. Suppose $\bar{u}\in \mathcal{M}$ is a constrained critical point of $\mathcal{E}\vert_{\mathcal{M}}$ with $g^{\prime}(\bar{u})\not\equiv 0$. Then $\mathcal{M}$ is locally a submanifold near $\bar{u}$ { and $\CalE\vert_{\mathcal{M}}$ }satisfies a {\L}ojasiewicz--Simon gradient inequality, i.e. there exist $C, \sigma>0$ and $\theta\in (0,\frac{1}{2}]$ such that for any $u\in \mathcal{M}$ with $\Norm{u-\bar{u}}{W^{2,2}(\Omega)}\leq \sigma$ we have
	\begin{align}\label{eq:AllenCahnLojaEstimate}
	\abs{\mathcal{E}(u)-\mathcal{E}(\bar{u})}^{1-\theta}\leq C\Norm{-\Delta u + W^{\prime}(u) - \frac{\langle -\Delta u+ W^{\prime}(u), g^{\prime}(u)\rangle_{L^2(\Omega)}}{\Norm{g^{\prime}(u)}{L^2(\Omega)}^2} g^{\prime}(u) }{L^2(\Omega)}.
	\end{align}
\end{thm}
\begin{proof}
	We can use the Hilbert space version of our main result, \Cref{cor:main hilbert}, with $U=V= W^{2,2}(\Omega)\cap W^{1,2}_0(\Omega)$ and $H=L^2(\Omega)$. Clearly, $V\hookrightarrow H$ densely. Moreover, we have $\nabla\mathcal{E}(u)= -\Delta u + W^{\prime}(u)$. Since $W^{\prime}\in \mathcal{C}^{\omega}(\R;\R)$, $V\hookrightarrow \mathcal{C}(\overline{\Omega})$ as $d\leq 3$ and $\mathcal{C}(\overline{\Omega}) \hookrightarrow H$, we conclude from \Cref{thm:nemytskiiAnalytic} that assumption (ii) in \Cref{cor:main hilbert} is satisfied. For $v\in V$, we have
	\begin{align*}
	(\nabla\mathcal{E})^{\prime}(\bar{u})v &= \dtzero \nabla\mathcal{E}(\bar{u}+tv) = -\Delta v  + \dtzero W^{\prime}(\bar{u}+tv) = -\Delta v + W^{\prime\prime}(\bar{u}) v.
	\end{align*}
	Thus, using \cite[Theorem 7.26]{GilbargTrudinger}, $(\nabla\mathcal{E})^{\prime}(\bar{u}) \colon V\to H$ is a compact perturbation of the Dirichlet-Laplacian $-\Delta\colon V\to H$, which is an isomorphism by standard elliptic theory \cite[Theorem 9.15]{GilbargTrudinger}. Thus, $(\nabla\mathcal{E})^{\prime}(\bar{u})$ is Fredholm of index zero by \Cref{prop:Fredholmproperties}. It follows from \Cref{lem:Ganalytic} that the $H$-gradient of $\mathcal{G}$ is given by $\nabla\mathcal{G}(u) = g^{\prime}(u)$, so assumption (iv) in \Cref{cor:main hilbert} is clearly satisfied. Similar to \Cref{lem:G''compact}, the Fréchet derivative $(\nabla\mathcal{G})^{\prime}(\bar{u})v = g^{\prime\prime}(\bar{u}) v$ is compact. Moreover, the assumption $g^{\prime}(\bar{u})\not\equiv 0$ means that assumption (vi) in \Cref{cor:main hilbert} is satisfied. Thus, \eqref{eq:AllenCahnLojaEstimate} follows from \eqref{eq:LSHilbert} and \Cref{rem:HilbertSpace m=1}.
\end{proof}
\begin{rem}
	The condition $d\leq 3$ is only needed to prove analyticity of the energy on $W^{2,2}(\Omega)\cap W^{1,2}_0(\Omega)$. However, like in \Cref{subsec:isoperimetric} one can consider different spaces to deal with the higher dimensional cases.
\end{rem}
\subsection{Area of surfaces of revolution with prescribed volume}\label{subsec:CMCRevolution}
In this subsection, we will discuss an application of \Cref{thm:main} to the area of a surface of revolution with prescribed boundary and prescribed inclosed volume. To that end, let $I=[a,b]\subset \R$ be an interval, $V\defeq W^{2,2}(I)\cap W^{1,2}_0(I)$, $H\defeq L^2(I)$ and consider
\begin{align*}
\mathcal{E}(u)\defeq 2\pi \int_{I} (1+u)\sqrt{1+(u^{\prime})^2} \diff x.
\end{align*}
This is the area of the surface of revolution $S$ obtained by rotating the graph of $1+u$ around the $x$-axis. Note that by requiring $u\in V$, we impose symmetric boundary conditions $1+u(a)=1+u(b)=1$. In order to ensure that $S$ is indeed a surface, we will study $\mathcal{E}$ on the set $U\defeq \{ u\in V\mid 1+u>0 \text{ on } {I}\}$.  Note that $U\subset V$ is open by the Sobolev embedding $W^{2,2}(I)\hookrightarrow \mathcal{C}(\overline{I})$.

We prescribe the volume inside $S$ by some fixed value $\nu\in \R$, i.e. we require
\begin{align*}
\mathcal{G}(u) \defeq \pi\int_{I}(1+u)^2\diff x -\nu =0.
\end{align*}
Unlike in \Cref{subsec:isoperimetric}, here we can work in the Hilbert space framework, since the embedding $W^{2,2}(I)\hookrightarrow \mathcal{C}^1(\overline{I})$ ensures that the energy $\mathcal{E}$ is analytic on $U$.
\begin{lem}
	The energies $\mathcal{E}$ and $\mathcal{G}$ are analytic on $U\subset V$. Moreover, for $u\in U, v\in V$ we have
	\begin{align}\label{eq:ExampleCMCrotE'}
	\mathcal{E}^{\prime}(u)v &= \langle \nabla\mathcal{E}(u),v\rangle = \int_{I} 2\pi \left[\sqrt{1+(u^{\prime})^2} - (1+u)\left(\frac{u^{\prime}}{\sqrt{1+(u^{\prime})^{2}}}\right)^{\prime}~\right]v\diff x\\
	\mathcal{G}^{\prime}(u)v &= \langle\nabla\mathcal{G}(u),v\rangle = \int_{I}\label{eq:ExampleCMCrotG'}  2\pi(1+u)v\diff x.
	\end{align}
	Moreover, the maps $U\ni u\mapsto\nabla\mathcal{E}(u)\in H$ and $U\ni u \mapsto\nabla\mathcal{G}(u)\in H$ are analytic.
\end{lem}
\begin{proof}
	Similar to \Cref{lem:EGAnalytic,lem:Ganalytic,lem:E'analytic}.
\end{proof}
\begin{lem}
	The Fréchet derivative $(\nabla\mathcal{E})^{\prime}(u)\colon V\to H$ is Fredholm of index zero.
\end{lem}
\begin{proof}
	Similar to \Cref{lem:ExampleE''Fredholm}, a short computation yields for $u\in U$ and $v\in V$
	\begin{align*}
	(\nabla\mathcal{E})^{\prime}(u) v = \frac{2\pi(1+u)}{(1+(u^{\prime})^2)^{\frac{3}{2}}} v^{\prime\prime} + K v \eqdef -Av + Kv,
	\end{align*}
	where $K\colon V\to H$ only contains terms in $v$ of order 1 or lower and is hence compact by \cite[Theorem 7.26]{GilbargTrudinger}. Note that like in \Cref{lem:ExampleE''Fredholm}, $A$ is an elliptic operator in $v$ by the embeddings $V\hookrightarrow \mathcal{C}^1(\overline{I})\hookrightarrow \mathcal{C}(\overline{I})$ and the requirement $u\in U$.	
	Thus, $\mathcal{E}^{\prime\prime}(u)$ is Fredholm of index zero by \Cref{prop:Fredholmproperties}.
\end{proof}
\begin{lem}
	The operator $(\nabla\mathcal{G})^{\prime}(u)\colon V\to H$ is compact for all $u\in U$.
\end{lem}
\begin{proof}
	The statement follows with the same ideas as in \Cref{lem:G''compact}.
\end{proof}
Consequently, similar to \Cref{thm:ExampleLoja,thm:AllenCahnLoja} we get the following result.
\begin{thm}
	Let $\mathcal{M}\defeq \{u\in U\mid \mathcal{G}(u)=0\}$. Suppose $\bar{u}\in \mathcal{M}$ is a constrained critical point of $\mathcal{E}\vert_{\mathcal{M}}$. Then $\mathcal{M}$ is locally a submanifold near $\bar{u}$ and satisfies a {\L}ojasiewicz--Simon gradient inequality, i.e. there exist $C, \sigma>0$ and $\theta\in (0,\frac{1}{2}]$ such that for any $u\in \mathcal{M}$ with $\norm{u-\bar{u}}{W^{2,2}(I)}\leq \sigma$ we have
	\begin{align}\label{eq:LSExampleCMCrevolution}
	\abs{\mathcal{E}(u)-\mathcal{E}(\bar{u})}^{1-\theta}\leq C  \Bigg\Vert2\pi\sqrt{1+(u^{\prime})^2} - 2\pi (1+u)\left(\frac{u^{\prime}}{\sqrt{1+(u^{\prime})^{2}}}\right)^{\prime} - \lambda(u) 2\pi(1+u)\Bigg\Vert_{L^2(I)},
	\end{align}
	where the scalar $\lambda(u)$ is given by
	\begin{align}\label{eq:ExampleCMCrotLambda}
	\lambda(u)= \frac{3}{2\nu} \mathcal{E}(u) - \left.\frac{\pi}{\nu} \frac{u^{\prime}}{\sqrt{1+(u^{\prime})^{2}}}\right\vert_{a}^{b} - \frac{2\pi}{\nu} \int_{I} \frac{1+u}{\sqrt{1+(u^{\prime})^{2}}} \diff x.
	\end{align}
\end{thm}
\begin{proof}
	We use \Cref{cor:main hilbert}. It remains to check condition (vi). Clearly, $\mathcal{G}(\bar{u})=0$ and we have $\nabla\mathcal{G}(\bar{u}) = 2\pi(1+\bar{u})\not\equiv0\in H=L^{2}(I)$, since $\bar{u}\in U$. Thus $\mathcal{E}\vert_{\mathcal{M}}$ satisfies a {\L}ojasiewicz--Simon gradient inequality at $\bar{u}$. By \Cref{rem:HilbertSpace m=1}, we get 
	\begin{align*}
	\abs{\mathcal{E}(u)-\mathcal{E}(\bar{u})}^{1-\theta}\leq  C \Big\Vert\nabla\mathcal{E}(u) - \frac{\langle\nabla\mathcal{G}(u), \nabla\mathcal{E}(u)\rangle}{\Norm{\nabla\mathcal{G}(u)}{L^2(I)}^2} \nabla\mathcal{G}(u)\Big\Vert_{L^2(I)}.
	\end{align*}
	Using \eqref{eq:ExampleCMCrotE'} and \eqref{eq:ExampleCMCrotG'}  yields \eqref{eq:LSExampleCMCrevolution} with $\lambda(u) = \frac{\langle\nabla\mathcal{E}(u), \nabla\mathcal{G}(u)\rangle}{\norm{\nabla\mathcal{G}(u)}{L^2(I)}^2}$. The explicit formula \eqref{eq:ExampleCMCrotLambda} for the scalar $\lambda$ can be proven using integration by parts and $\mathcal{G}(u)=0$.
\end{proof}

\appendix
\renewcommand{\thesection}{\Alph{section}}
\section{Proof of \texorpdfstring{\Cref{thm:FredholmLoja}}{Theorem 1.2}}\label{subsec:AppendixLoja}

Following the notation in \cite[Section 3]{Chill03}, let $V$ be a Banach space,  $\mathcal{E}\in \mathcal{C}^2(U;\R)$ with $U\subset V$ open and let $\varphi\in U$ be a critical point of $\mathcal{E}$. Then $M\defeq\mathcal{E}^{\prime}\in \mathcal{C}^1(U;V^{\ast})$ and $L\defeq \mathcal{E}^{\prime\prime}\in \mathcal{C}(U;\mathcal{L}(V,V^{\ast}))$. In \cite{Chill03}, the {\L}ojasiewicz--Simon gradient inequality is proven under the following assumptions.
\begin{hyp}[Hypothesis 3.2. in {\cite{Chill03}}]\label{hyp3.2}
	The kernel $V_0:= \ker L(\varphi)$ is complemented, i.e. there exists a projection $Q\in \mathcal{L}(V)$ such that $\Image Q = V_0$.
\end{hyp}
Clearly, under the assumption of \Cref{hyp3.2}, we have $V = V_0\oplus V_1$ with $V_1 = \ker Q$, and similarly $V^{\ast} = V_0^{\ast}\oplus V_1^{\ast}$ with $V_0^{\ast} = \Image Q^{\ast}$ and $V_1^{\ast} = \ker Q^{\ast}$. Note that this abuse of notation is justified since $V_1^{\ast} = \{ v^{\ast}\in V^{\ast}\mid v^{\ast}(w) = 0 \text{ for all }w\in V_0\}$ is isomorphic to the dual of $V_0$ and similarly for $V_0^{\ast}$.
\begin{hyp}[Hypothesis 3.4. in {\cite{Chill03}}]\label{hyp3.4}
	There exists a Banach space $W$ with the following properties.
	\begin{enumerate}[(i)]
		\item $W\hookrightarrow V^{\ast}$ continuously,
		\item the adjoint $Q^{\ast}\in \mathcal{L}(V^{\ast})$ of the projection in \Cref{hyp3.2} leaves $W$ invariant,
		\item $M\in\mathcal{C}^{1}(U;W)$,
		\item $\Image L(\varphi) = V_1^{\ast}\cap W$.
	\end{enumerate}
\end{hyp}
\begin{thm}[Corollary 3.11 in {\cite{Chill03}}]\label{cor:3.11}
	Suppose $\varphi\in U$ is a critical point of ${\mathcal{E}}$ and assume \Cref{hyp3.2,hyp3.4} hold. Assume that there exists Banach spaces $X\subset V$ and $Y\subset W$ such that
	\begin{enumerate}[(i)]
		\item the spaces $X$ and $Y$ are invariant under $Q$ and $Q^{\ast}$, respectively,
		\item the restriction of the derivative $M$ to $U\cap X$ is analytic in a neighborhood of $\varphi$ with values in $Y$,
		\item $\ker L(\varphi)$ is contained in $X$ and finite-dimensional,
		\item $\Image L(\varphi)\vert_{X} = \ker Q^{\ast}\cap Y$.
	\end{enumerate}
	Then ${\mathcal{E}}$ satisfies the {\L}ojasiewicz--Simon gradient inequality at $\varphi$.
\end{thm}
\begin{proof}[{Proof of \Cref{thm:FredholmLoja}}]
	We will show that \Cref{hyp3.2,hyp3.4} and the assumptions of \Cref{cor:3.11} are satisfied for $\varphi\defeq\bar{u}$. We follow \cite[Appendix A]{DPS16}, where the result was proven for a special case. Let $Z$ as in \Cref{thm:FredholmLoja}, then (i) in \Cref{hyp3.4} is satisfied for $W\defeq Z^{\ast}$ since $V\hookrightarrow Z$ densely.

	We set $X\defeq V, Y\defeq W = Z^{\ast}$.
	Since $L(\varphi) = \mathcal{E}^{\prime\prime}(\bar{u})\colon V\to Z^{\ast}$ is Fredholm by assumption (iii), its kernel $V_0 \defeq \ker L(\varphi) \subset V\subset Z$ is finite-dimensional, thus (iii) in \Cref{cor:3.11} holds. Let $d\defeq \dim V_0<\infty$ and note that $V_0$ is closed in both $V$ and $Z$ and thus complemented in both spaces (cf. \Cref{lem:fin_co_dim complemented}). This implies that there exists $V_1\subset V$ closed such that $V=V_0\oplus V_1$ and there exists a projection $\tilde{Q}\in \mathcal{L}(Z)$ onto $V_0$. We will extend this to obtain a particular projection on $Z$ onto $V_0$. Note that ${Q \defeq \tilde{Q}\vert_{V}\colon V\to V}$ is also continuous, since as $V_0 = \Image \tilde{Q} = \Image Q$ is finite-dimensional, there exist $C, C^{\prime}, C^{\prime\prime}>0$ such that $\Norm{Qv}{V}\leq C\Vert{\tilde{Q}v}\Vert_{Z}\leq {C}^{\prime}\Norm{v}{Z}\leq C^{\prime\prime} \Norm{v}{V}$ for all $v\in V$, using that $V\hookrightarrow Z$. Now, $Q\in \mathcal{L}(V)$ satisfies \Cref{hyp3.2}. Denote by $Q^{\ast}\in \mathcal{L}(V^{\ast})$ the adjoint of $Q$.
	
	Assumption (ii) in \Cref{thm:FredholmLoja} immediately implies that assumption (iii) in \Cref{hyp3.4} and (ii) in \Cref{cor:3.11} are satisfied.
	
	In order to prove \Cref{hyp3.4} (iv), recall that by Schwarz's Theorem (cf. \cite[XIII, Theorem 5.3]{Lang12}), the second derivative $L(\varphi)\colon V\to Z^{\ast}$ is symmetric, i.e. for $v,w\in V$ we have $(L(\varphi)v)(w)= (L(\varphi)w)(v)$. Thus, for $v\in V$ and for any $w\in V_0 = \ker L(\varphi)$, we have $(L(\varphi)v)(w)= 0$, i.e. 
	\begin{align}\label{eq:inclusionImageL}
	\Image L(\varphi) \subset \{ z^{\ast}\in Z^{\ast} \mid z^{\ast}(v) =0 \text{ for all } v\in V_0\}=V_1^{\ast}\cap Z^{\ast},
	\end{align}
	using $V^{\ast} = V_0^{\ast}\oplus V_1^{\ast}$.
	As a next step, we show that the inclusion in \eqref{eq:inclusionImageL} is an equality. Since  $\ind{L(\varphi)}=0$, we have $\codim(\Image L(\varphi), Z^{\ast}) = \dim\ker L(\varphi) = \dim V_0=d$. Moreover, we have $\codim (V_1^{\ast}\cap Z^{\ast}, Z^{\ast})=d$. Indeed, considering $V_0$ as a subspace of $Z$, by \eqref{eq:inclusionImageL}, we have $V_1^{\ast}\cap Z^{\ast} = V_0^{\perp}$, the annihilator of $V_0$ in $Z^{\ast}$. By \cite[Proposition 11.13]{Brezis10}, we have $d=\dim V_0 = \codim V_0^{\perp} = \codim(V_1^{\ast}\cap Z^{\ast}, Z^{\ast})$. Hence, $\Image L (\varphi)$ and $V_1^{\ast}\cap Z^{\ast}$ are two subspaces with the same finite codimension $d$ in $Z^{\ast}$, with one contained in the other. Therefore, they have to be equal by \cite[Proposition 5 in II §7. 3.]{Bourbaki98}. Thus, equality holds in \eqref{eq:inclusionImageL}, so \Cref{hyp3.4} (iv) and assumption (iv) in \Cref{cor:3.11} are satisfied.
	
	It remains to check that $Q^{\ast}$ leaves $W=Z^{\ast}$ invariant. Let $z^{\ast}\in Z^{\ast}$ and $v\in Z$. Then 
	$\left( Q^{\ast} z^{\ast}\right) v = z^{\ast}\left(Qv\right)$ is linear in $v$ and for  $C = \norm{z^{\ast}}{Z^{\ast}}\norm{\tilde{Q}}{\mathcal{L}(Z)}\geq 0$, we have
	\begin{align}\label{eq:QvEstimate}
	\abs{z^{\ast}(Qv)}\leq \Norm{z^{\ast}}{Z^{\ast}} \Norm{Qv}{Z} = \Norm{z^{\ast}}{Z^{\ast}} \norm{\tilde{Q}v}{Z} \leq C \Norm{v}{Z}.
	\end{align}
	By \eqref{eq:QvEstimate}, $Q^{\ast}z^{\ast}\colon Z\to \R$ is an element of the dual space $Z^{\ast}$. This yields that \Cref{hyp3.4} (ii) and assumption (i) in \Cref{cor:3.11} are satisfied. Hence, we may apply \Cref{cor:3.11} to conclude that $\mathcal{E}$ satisfies a {\L}ojasiewicz--Simon gradient inequality in a neighborhood of $\varphi=\bar{u}$. 
\end{proof}
\section{Basic sequences in Banach spaces}\label{sec:BasicSeq}
The goal of this subsection is to provide a proof of \Cref{thm:ExistBasicSeq}. We therefore use a generalization of the notion of an orthonormal basis in a Hilbert space to the Banach space situation. We follow the presentation of \cite[Chapter 1]{AlbiacKalton06}.
\begin{defn}[{Cf. \cite[Definition 1.1.2 and Definition 1.1.5]{AlbiacKalton06}}]\label{def:SchauderBasis}
	Let $X$ be a Banach space and $(e_n)_{n\in \N}\subset X$ be a sequence. Suppose there exists a sequence $(\phi_k)_{k\in \N}\subset X^{\ast}$ such that
	\begin{enumerate}[(i)]
		\item  $\phi_k(e_n)=\delta_{k,n}$ for all $k, j\in \N$,
		\item $v = \sum_{k=1}^{\infty} \phi_k(v) e_k$ for all $v\in X$.
	\end{enumerate}
	Then, $(e_n)_{n\in\N}$ is called a \emph{Schauder basis} for $X$ with \emph{associated biorthogonal functionals} $(\phi_k)_{k\in \N}$. A sequence $(e_n)_{n\in N}$ in a Banach space $X$ is called a \emph{basic sequence} if it is a Schauder basis for $\overline{\operatorname{span}}\{e_n\mid n\in \N\}$.
\end{defn}
The following is an immediate consequence of the Uniform Boundedness Principle.
\begin{prop}[Cf. {\cite[Proposition 1.1.4]{AlbiacKalton06}}]
	Let $(e_n)_{n\in N}$ be a Schauder basis for a Banach space $X$. Then the \emph{natural projections} $S_N \colon X\to X, S_N v \defeq \sum_{k=1}^N \phi_k(v)e_k$ are uniformly bounded in $\mathcal{L}(X)$, i.e. we have
	\begin{align}\label{eq:defK}
	K \defeq \sup_{N\in \N}\Norm{S_N}{} <\infty.
	\end{align}
\end{prop}
The number $K$ is called the \emph{basis constant} of the sequence $(e_n)_{n\in \N}$. The following existence result is what we need to prove \Cref{thm:ExistBasicSeq}.
\begin{thm}\label{thm:ExistBasicSeq0}
	Let $V$ be an infinite-dimensional Banach space and let $\varepsilon>0$. Then there exists a basic sequence $(e_n)_{n\in \N}$ with basis constant $K\leq 1+\varepsilon$ and $\norm{e_n}{}=1$ for all $n\in \N$.
\end{thm}
\begin{proof}
	The existence of a basic sequence $(e_n)_{n\in \N}$ with basis constant less than $1+\varepsilon$ is exactly the statement of \cite[Corollary 1.5.3]{AlbiacKalton06}. Investigating its proof, we note that the $e_n$ are chosen from $S\defeq \{x\in X\mid \Norm{x}{}=1\}$, which proves the claim.
\end{proof}
Now, we are finally able to prove \Cref{thm:ExistBasicSeq}.
\begin{proof}[{Proof of \Cref{thm:ExistBasicSeq}}]
	Using \Cref{thm:ExistBasicSeq0}, we obtain a sequence $(e_n)_{n\in \N}$ with $\norm{e_n}{}=1$ for all $n\in \N$ which is a basic sequence, i.e. a Schauder basis for the Banach space $X\defeq \overline{\operatorname{span}}\{e_n\mid n\in \N\}\subset V$. By \Cref{def:SchauderBasis} there exists an associated biorthogonal sequence $(\phi_k)_{k\in\N}\subset X^{\ast}$. Note that by \Cref{thm:ExistBasicSeq0}, the basis constant $K$ defined in \eqref{eq:defK} satisfies $K \leq 1+\varepsilon$. For $x\in X$ and $k\in \N$, we have setting $S_0 \defeq 0 \in \mathcal{L}(X)$
	\begin{align*}
	\abs{\phi_k(x)} = \norm{\phi_k(x)e_k}{} = \norm{S_k x - S_{k-1}x}{} \leq 2 K \norm{x}{}, 
	\end{align*}
	thus $\norm{\phi_k}{X^{\ast}} \leq 2K \leq 2(1+\varepsilon)$. By the Hahn--Banach Theorem, there exist extensions of $\phi_k \colon X \to \R$ also denoted $\phi_k \colon V \to \R$ with $\norm{\phi_k}{V^{\ast}} = \norm{\phi_k}{X^{\ast}} \leq 2(1+\varepsilon)$.
\end{proof}

\section*{Acknowledgments}
The author would like to thank Anna Dall’Acqua, Marius Müller and Adrian Spener for helpful discussions and comments.
\section*{Funding}
This work was supported by the Deutsche Forschungsgemeinschaft (DFG, German Research Foundation)-Projektnummer: 404870139.
 
\bibliographystyle{abbrv}
\bibliography{Literatur}

\end{document}